\documentclass{article}

\usepackage[english]{babel}

\usepackage[colorlinks=true, allcolors=blue]{hyperref}

\usepackage{graphicx}
\usepackage{bm}
\usepackage{amsmath}
\usepackage{amssymb}
\usepackage{amsfonts}
\usepackage{enumitem}
\usepackage{mathrsfs}
\usepackage{xcolor}
\usepackage{cleveref}
\usepackage{mathtools}
\usepackage{tikz}
\usepackage{pgf}
\usepackage{adjustbox}
\usepackage{comment}
\usepackage{natbib}
\let\cite\citep
\usepackage{thm-restate}
\usepackage{amsthm}

\usepackage[margin=1in]{geometry}

\newcommand{\pr}[0]{\mathrm{pr}}
\newcommand{\pos}[0]{\mathrm{pos}}
\newcommand{\obs}[0]{\mathrm{obs}}
\newcommand{\unknown}{\mbf{f}}
\newcommand{\data}{\mbf{y}}
\newcommand{\dataRV}{\mbf{Y}}
\newcommand{\noise}{\pmb{\epsilon}}
\newcommand{\fwdOper}[0]{\mbf{G}}
\newcommand{\hatFwdOper}[0]{\widehat{\mbf{G}}}
\newcommand{\hatFwdOperOLR}[0]{\widehat{\mbf{G}}^{\text{\tiny OLR}}}
\newcommand{\hatFwdOperROM}[0]{\widehat{\mbf{G}}^{\text{\tiny ROM}}}

\newcommand{\PrMn}{\pmb{\mu}_{\pr}}

\newcommand{\PrCov}{\mathbf\Gamma_{\pr}}
\newcommand{\rootPrCov}{\mathbf{S}_{\pr}}

\newcommand{\ObsCov}{\mathbf\Gamma_{\obs}}
\newcommand{\rootObsCov}{\mathbf{S}_{\obs}}

\newcommand{\PosMn}{\pmb{\mu}_{\pos}}
\newcommand{\hatPosMn}{\widehat{\pmb{\mu}}_{\pos}}
\newcommand{\hatPosMnOLR}[0]{\widehat{\pmb{\mu}}_{\pos}^{\text{\tiny OLR}}}
\newcommand{\hatPosMnROM}[0]{\widehat{\pmb{\mu}}_{\pos}^{\text{\tiny ROM}}}
\newcommand{\PosCov}{\mathbf\Gamma_{\pos}}
\newcommand{\hatPosCov}{\widehat{\mathbf\Gamma}_{\pos}}
\newcommand{\hatPosCovOLR}{\widehat{\mathbf\Gamma}_{\pos}^{\text{\tiny OLR}}}
\newcommand{\hatPosCovROM}{\widehat{\mathbf\Gamma}_{\pos}^{\text{\tiny ROM}}}

\newcommand{\ObsOper}{\mathbf{C}}

\newcommand{\fullModel}{\mathbf{K}}
\newcommand{\redModel}{\widehat{\mathbf{K}}}
\newcommand{\state}{\mathbf{u}}

\newcommand{\lSgMat}{\bm{\Omega}}
\newcommand{\lSgVec}{\bm{\omega}}
\newcommand{\rSgMat}{\bm{\Phi}}
\newcommand{\rSgVec}{\bm{\phi}}
\newcommand{\SgValMat}{\bm{\Delta}}
\newcommand{\SgVal}{\delta}

\newcommand{\testMat}{\mathbf{V}}
\newcommand{\trialMat}{\mathbf{W}}

\newcommand{\rankPPH}{\mathfrak{r}}

\newcommand{\Id}{\mathcal{I}}
\newcommand{\R}[0]{\mathbb{R}}
\newcommand{\N}[0]{\mathbb{N}}
\newcommand{\mbf}[1]{\mathbf{#1}}
\newcommand{\norm}[1]{\ensuremath{\lVert #1 \rVert}}
\newcommand{\Norm}[1]{\ensuremath{\left\lVert #1 \right\rVert}}

\newcommand{\KLD}[0]{\textup{D}_{\textup{KL}}}

\newcommand*{\todo}[1]{\bgroup\color{red}TODO: #1\egroup}

\newtheorem{theorem}{Theorem}[section]
\newtheorem{lemma}[theorem]{Lemma}
\newtheorem{proposition}[theorem]{Proposition}

\theoremstyle{definition}

\newtheorem{remark}[theorem]{Remark}

\newtheorem{assumption}[theorem]{Assumption}

\newcommand{\Crefprop}[1]{Proposition~\ref{#1}}
\newcommand{\Crefrem}[1]{Remark~\ref{#1}}
\newcommand{\Crefass}[1]{Assumption~\ref{#1}}
\newcommand{\Creflem}[1]{Lemma~\ref{#1}}
\setenumerate[1]{label=(\roman*)}

\usepackage{amsopn}

\DeclareMathOperator*{\range}{ran}

\DeclareMathOperator*{\rank}{rank}
\DeclareMathOperator*{\mathdim}{dim}
\DeclareMathOperator*{\argmin}{argmin}

\numberwithin{equation}{section}

\title{Error bounds for approximate posteriors from likelihood-informed reduced-order models}
\author{Han Cheng Lie\thanks{~Institut f\"ur Mathematik, Universit\"at Potsdam, Germany ({\tt han.lie@uni-potsdam.de}),  ORCID ID: 0000-0002-6905-9903} \and Jakob Scheffels\thanks{Engineering Risk Analysis Group, TUM School of Engineering and Design, and Institute for Advanced Study, Technical University of Munich, Germany ({\tt jakob.scheffels@tum.de}), ORCID ID: 0000-0002-9589-0426}  \and Elisabeth Ullmann\thanks{~Department of Mathematics, TUM School of Computation, Information and Technology, Technical University of Munich, Germany ({\tt elisabeth.ullmann@ma.tum.de}), ORCID ID: 0000-0002-5699-7394}}

\begin{document}

\maketitle

\begin{abstract}
In the design of computational methods for Bayesian inverse problems, costly forward model evaluations make it difficult to sample from or compute the posterior.
This motivates the need for approximate forward models that are cheaper to evaluate.
We consider reduced-order forward models which exploit the lower-dimensional structure in the Bayesian inverse problem by projecting to the `likelihood-informed subspace' of the parameter space where the prior-to-posterior update is significant.
However, the theoretical properties of these reduced-order forward models and their impact on the solution of the Baysian inverse problem are not always well-understood.
In this work we consider linear Gaussian inverse problems with a possibly singular prior covariance matrix.
We analyse a recently proposed reduced-order model which uses a Petrov-Galerkin projection to likelihood-informed subspaces that arise in optimal low-rank approximations of the posterior covariance matrix.
We bound the error in the resulting approximation of the root prior-preconditioned Hessian of the data misfit.
Based on this we also bound the errors of the approximate posterior covariance and mean.
Our analysis shows that this reduced-order model recovers the exact posterior when the rank of the reduced-order model is equal to the `intrinsic dimension' of the inverse problem, i.e. the rank of the prior-preconditioned Hessian.
Two numerical experiments from structural engineering illustrate the performance of our bounds.
\end{abstract}

\section{Introduction}
\label{section:introduction}
In the Bayesian approach to inverse problems, the parameter and the data are treated as random variables, where the unknown parameter is endowed with a prior distribution.
The inference takes the form of the Bayesian prior-to-posterior update, by integrating the likelihood against the prior distribution.
Sampling from or characterising the posterior requires evaluating the forward model, i.e. the function that maps parameters to measurements before they are contaminated by noise.
In many applications in science and engineering, evaluating the forward model involves evaluating a `parameter-to-state operator' or `solution operator' associated to the numerical solution of a partial differential equation (PDE); see e.g. ~\cite{stuartInverseProblemsBayesian2010b,cotterApproximationBayesianInverse2010a}.
Finer discretisations lead to higher-dimensional solution spaces, and thus to more expensive forward model evaluations, even if the forward model is linear.
This motivates the development of computationally cheaper approximate forward models.

An important insight is that Bayesian inverse problems often have some low-dimensional structure that can be exploited to reduce the computational cost of solving them.
For linear Gaussian inverse problems, i.e. Bayesian inverse problems where the forward model is linear and the prior and noise are Gaussian, this insight appears to have been first described in \cite{Flath2011} and mathematically analysed in \cite{Spantini2015}.
This low-dimensional structure can arise because of strongly concentrated priors, smoothing properties of the forward model, or measurements that are not informative due to noise contamination.
This low-dimensional structure is captured by the fact that the Bayesian prior-to-posterior update occurs mostly on a low-dimensional subspace of the parameter space, so that the posterior and prior do not differ significantly on the orthogonal complement of this subspace.
For linear Gaussian inverse problems, the low-dimensional structure is often characterised by the spectral decay of the so-called `prior-preconditioned Hessian', and the subspace or a suitable transformation of it is referred to as the `likelihood-informed subspace' (LIS).
In particular, in \cite{Spantini2015}, it was shown that by precomposing the forward model with a projector to the LIS, one can restrict the original, possibly high-dimensional inverse problem to an inverse problem on the LIS and thereby obtain an optimal approximation of the Gaussian posterior.
However, because this approach still requires the evaluation of the forward model and thus of the PDE solution operator, it does not address the problem of expensive forward model evaluations.

In \cite{Scheffels2025}, an approach is proposed to address this problem by replacing the solution operator with a reduced-order model that is constructed by Petrov-Galerkin projection of the PDE solution operator, where the test and trial bases of the Petrov-Galerkin projection are related to the LIS.
In this approach, the reduced-order model of the PDE solution operator is used instead of the original solution operator, in order to obtain an approximation of the forward model that is cheaper to evaluate than the original forward model.
Using the approximate forward model leads to an approximation of the exact posterior corresponding to the original forward model.
This raises the need for a theoretical analysis of the error of the approximate posterior.

The first contribution of this paper is to generalise some well-known results concerning optimal low-rank approximations of linear Gaussian inverse problems, from the original setting of invertible prior covariances in \cite{Spantini2015} to the setting of possibly singular prior covariances.
This generalisation is important, because in some applications such as structural engineering, it may be unnecessary or undesirable to account for uncertainty in every degree of freedom in the parameter space.
While the setting of possibly singular prior covariances has been studied before, e.g. in \cite{konigDimensionModelReduction2025,KoenigLie2025}, our analysis differs from the analysis we have found in the literature.

The main novel contribution of this paper is a theoretical analysis of the approximate posterior associated to the likelihood-informed reduced-order models of \cite{Scheffels2025}.
To achieve this, we analyse the root prior-preconditioned Hessian associated to the approximate forward model.
In particular, we show in \Creflem{lemma:decomposition_of_root_of_approximate_PPH} that the root prior-preconditioned Hessian associated to the approximate forward model agrees with the root prior-preconditioned Hessian associated to the \emph{exact} forward model, when both are post-composed with the orthogonal projection to certain subspaces in the data space.
In our main result, \Cref{theorem:error_decomposition_approximate_root_PPH}, we use the preceding result to prove upper bounds on the errors in $p$-Schatten norms of the root prior-preconditioned Hessian, which in turn yield upper bounds for the approximate posterior covariance and the approximate posterior mean associated to the approximate forward model.
This suffices to analyse the error of the approximate posterior, because the exact and approximate posteriors are Gaussian in our setting.
Finally, we provide two numerical examples of inverse problems from structural engineering, to illustrate the effectiveness of our bounds.

The first key insight of this work is that the reduced-order model proposed in \cite{Scheffels2025} correctly identifies the subspace on which the prior-to-posterior update occurs.
The second key insight is that the reduced-order model identifies the `intrinsic dimension' of the inverse problem, in the common setting where the dimension of the data is smaller than the dimension of the uncertain component of the parameter.
We quantify the dimension of the uncertain component by the rank of the prior covariance.
In particular, the approximate posterior that results from using the reduced-order model agrees with the exact posterior, when the dimension of the projection is equal to the dimension of the data.

\subsection{Related literature}
\label{subsection:related_literature}
Various surrogate modeling methods have been used in Bayesian inverse problems. Polynomial chaos expansions approximate the forward model in the span of a truncated set of cheap-to-evaluate basis functions~\cite{AdaptiveMultifidelityPCE2019,novakPhysicsinformedPolynomialChaos2024}.
However, the truncation may introduce errors due to the inexpressivity of the basis~\cite{LimitationsPolynomialChaos2015}.
Another approach constructs the surrogate using universal function approximators such as Gaussian processes and neural networks~\cite{dinkelSolvingBayesianInverse2023, villaniAdaptiveGaussianProcess2024, yanAdaptiveSurrogateModeling2020, deveneyDeepSurrogateApproach2021,pasparakisBayesianNeuralNetworks2025, pfortnerPhysicsInformedGaussianProcess2024}. These typically require a larger volume of training data to achieve accurate results, which may not be available for certain applications.

Projection-based model reduction constructs surrogates by projecting the differential equation that describes the inverse problem onto a low-dimensional subspace \cite{Antoulas2005, benner2015survey}. Common methods to construct the reduced model are proper orthogonal decomposition ~\cite{lumley1981coherent,sirovich1987turbulence, ghattasLearningPhysicsbasedModels2021a, nguyenModelOrderReduction2014, cuiDatadrivenModelReduction2015, xiongAcceleratingBayesianInference2021, raoInverseParameterEstimation2024} and reduced basis methods~\cite{hesthaven2022reduced,rozza2024short,chenSteinVariationalReduced2021,silvaReducedBasisEnsemble2023}. These methods approximate the system state in the span of principal components of available or generated snapshot data. The accuracy of the obtained reduced model depends on the choice of subspace onto which the equation is projected  \cite{eiermannGeometricAspectsTheory2001}. However, the reduced models do not consider the context of the inference and the low dimensionality of the transition from prior to posterior, but rather compare global accuracy while minimising the snapshot data that is used.

Some works have developed approximations for Bayesian inverse problems by exploiting their intrinsic low dimensionality due to sparse measurements and smoothing properties of the forward model. The posterior distribution is approximated by replacing the high-dimensional parameter space by the LIS~\cite{Spantini2015, Flath2011,bui-thanhExtremescaleUQBayesian2012, bui-thanhComputationalFrameworkInfiniteDimensional2013,cuiLikelihoodinformedDimensionReduction2014b,zahm2022certified}. For linear Gaussian inverse problems, the work \cite{Spantini2015} appears to be the first work to describe the optimality of the obtained approximations for finite-dimensional parameter spaces; this work has been recently generalised to Hilbert spaces \cite{CarereLie2025a,CarereLie2025b}.
Such methods can be considered `dimension reduction' methods, as they only restrict the parameter to the LIS and still require the evaluations of the high-dimensional forward model to evaluate the likelihood. These dimension reduction approaches have been further developed in the context of linear dynamical systems where the goal of the inference is the initial state~\cite{qianModelReductionLinearBalancing, konigDimensionModelReduction2025, konigTimeLimitedBalancedTruncation2023,stavrinidesEnsembleKalmanApproach2025, freitagInferenceOrientedBalancedTruncation2024a}.

In \cite{Scheffels2025}, the authors present a likelihood-informed model reduction method for the inference of static structural loads from linear observations of the state. The reduced model is obtained by Petrov-Galerkin projection onto the LIS, and numerical experiments showed accurate posterior approximation. However, no theoretical analysis of the method is given. In particular, while error assessments of the accuracy of the posterior approximations are performed for the experiments, no error bounds are given.

In \cite{sprungkLocalLipschitzStability2020a}, upper bounds for the error of approximate posteriors arising from approximations of the possibly nonlinear forward model are proven, in the setting where both the prior and the noise may be non-Gaussian.
Because the posterior is in general not Gaussian, statistical distances and divergences are used to quantify the error in the approximate posterior.
For forward models given as a composition of a linear state-to-observation map with a possibly nonlinear parameter-to-state map, \cite{Cvetkovic_etal_BayesOED_2024} proved error bounds with respect to the Kullback-Leibler divergence of approximate posteriors arising from approximations of the parameter-to-state map.
In \cite{KoenigLie2025}, a perturbation analysis is used in the context of linear Gaussian inverse problems, in order to obtain error bounds of the posterior mean and covariance approximations in terms of the error in the root prior-preconditioned Hessian. In addition, connections between such inverse problems and  system-theoretic model reduction methods based on balanced truncation are used to reformulate these error bounds in terms of Hankel singular values.

\subsection{Outline}
\label{subsection:outline}

After introducing the notation in \Cref{subsection:notation}, we describe the setting of our work in \Cref{section:setting}. We recall some facts from the theory of optimal low-rank approximations of linear Gaussian inverse problems in \Cref{subsection:optimal_low_rank_approximation}. Moreover, we extend some well-known properties of optimal low-rank approximations that were established in \cite{Spantini2015} for invertible prior covariances to the case of possibly singular prior covariances.
In \Cref{subsection:ROM}, we recall the construction of the reduced-order model that was defined in \cite{Scheffels2025}.
In \Cref{section:error_analysis}, we present the key theoretical results of this work.
We explore these theoretical results further using numerical experiments, which we describe in \Cref{section:numerical_experiments}.
We conclude in \Cref{section:conclusion} and provide detailed proofs in Appendix~\ref{section:proofs}.

\subsection{Notation}
\label{subsection:notation}

For $\mbf{A}\in\R^{d\times d}$, $d\in\N$, $\mbf{A}\succ\mbf{0}$ and $\mbf{A}\succeq \mbf{0}$ indicates that $\mbf{A}$ is symmetric positive definite (s.p.d.) and symmetric positive semidefinite (s.p.s.d.) respectively.
For $m,n\in\N$ and $\mbf{A}\in\R^{m\times n}$, we write $\mbf{A}^\top$, $\range{\mbf{A}}\coloneqq\{\mbf{A}\mbf{z}:\mbf{z}\in\R^n\}$ and $\rank{\mbf{A}}\coloneqq\mathdim{\range{\mbf{A}}}$ to denote the transpose, range and rank of $\mbf{A}$ respectively.
For $1\leq r\leq n$, $\mbf{A}_r\coloneqq [a_1|\cdots|a_r]\in\R^{m\times r}$ and $\mbf{A}_\bot\coloneqq[ a_{r+1}|\cdots|a_n]\in\R^{m\times (n-r)}$ denote the column submatrix of $\mbf{A}$ consisting of the leading $r$ columns and the trailing $n-r$ columns respectively.
If $r\geq n$, then we define $\mbf{A}_\bot\coloneqq \mbf{0}$, where the number of columns of $\mbf{0}$ may depend on the context.
For $1\leq p\leq \infty$ and $\mbf{x}\in\R^{m}$, $\norm{\mbf{x}}_p$ denotes the $\ell_p$ norm and $\norm{\mbf{A}}_p$ denotes the $p$-Schatten norm.
In particular, $\norm{\mbf{A}}_\infty$ denotes the spectral norm or operator norm $\norm{\mbf{A}}_\infty=\sup\{ \norm{\mbf{A}\mbf{x}}_2:\norm{\mbf{x}}_2=1\}$.
For $m\in\N$, we denote the $m\times m$ identity matrix by $\Id_m$.
For $a,b\in\R$, $a\wedge b\coloneqq \min\{a,b\}$.
We write $\text{diag}(s_i)_{i=1}^{n}$ to denote a possibly rectangular matrix whose diagonal entries are $(s_i)_{i=1}^{n}\in\R^n$ and $\mbf{x}\sim\mathcal{N}(\mbf{m},\mbf{C})$ to indicate that the random variable $\mbf{x}$ has the normal or Gaussian distribution with mean $\mbf{m}$ and covariance $\mbf{C}$.

\section{Setting}
\label{section:setting}

\begin{assumption}
\label{assumption_main}
{\rm The inverse problem is given by the observation model}
\begin{subequations}
\label{eq:inverse_problem}
\begin{align}
    \dataRV &=\fwdOper \unknown+\noise, \quad \text{s.t. }\,  \fullModel\state=\unknown,
    \label{eq:observation_model}
    \\
   \fwdOper &=\ObsOper\fullModel^{-1}.
   \label{eq:forward_operator}
\end{align}
\end{subequations}
{\rm In \eqref{eq:observation_model}, the $\R^m$-valued data random variable $\dataRV$ is the image under the forward model $\fwdOper\in\R^{m\times d}$ of the $\R^d$-valued unknown parameter $\unknown\sim \mathcal{N}(\PrMn,\PrCov)$ for some $\PrCov\succeq \mbf{0}$, contaminated by $\R^m$-valued noise $\noise\sim\mathcal{N}(\mbf{0},\ObsCov)$ for $\ObsCov\succ \mbf{0}$.
The equation $\fullModel\state=\unknown$ relates the unknown $\unknown$ to the corresponding `state' variable $\state$, by the invertible linear operator $\fullModel\in\R^{d\times d}$.
In \eqref{eq:forward_operator}, the forward model $\fwdOper$ is expressed as the composition of a linear state-to-observation operator $\ObsOper \in\R^{m\times d}$ with the linear parameter-to-state or solution operator $\fullModel^{-1}$.}
\end{assumption}

In \Crefass{assumption_main}, the prior covariance $\PrCov$ is s.p.s.d. but not necessarily s.p.d.. Note, however, that the observation noise covariance must be s.p.d..

The structural assumption \eqref{eq:forward_operator} on $\fwdOper$ often arises in the context of inverse problems based on differential equations with uncertain parameters: $\state=\fullModel^{-1}\unknown$ encodes the numerical solution $\state$ of a discretised boundary or initial value problem corresponding to the parameter value $\unknown$, and $\ObsOper$ encodes the measurement of the solution at finitely many locations in the spatial or spatiotemporal domain associated to the differential equation.
See e.g. \cite{Cvetkovic_etal_BayesOED_2024,Scheffels2025,konigDimensionModelReduction2025} for recent work where this structural assumption on $\fwdOper$ arises.

Given a realisation $\data$ of the data random variable $\dataRV$ in \eqref{eq:inverse_problem}, the exact posterior mean and covariance corresponding to the reference or `exact' forward model $\fwdOper$ are defined by
\begin{subequations}
\label{eq:posterior}
    \begin{align}
    \PosMn(\data) &=\PrMn+\PrCov \fwdOper^\top(\fwdOper\PrCov\fwdOper^\top+\ObsCov)^{-1}(\data-\fwdOper\PrMn)\in \R^{d},
    \label{eq:pos_mean}
    \\
    \PosCov &= \PrCov - \PrCov \fwdOper^\top(\fwdOper\PrCov \fwdOper^\top+\ObsCov+)^{-1}\fwdOper\PrCov\in \R^{d \times d}.
    \label{eq:pos_covariance}
\end{align}
\end{subequations}
By \Crefass{assumption_main}, the Fisher information matrix or the Hessian of the data misfit given by the observation model \eqref{eq:inverse_problem} is $\fwdOper^\top \ObsCov^{-1}\fwdOper$.
We shall henceforth use `Hessian' to refer to the Hessian of the data misfit.
The prior-preconditioned Hessian is given by $\PrCov^{1/2}\fwdOper^\top \ObsCov^{-1}\fwdOper\PrCov^{1/2}\in\R^{d\times d}$, where $\PrCov^{1/2}$ denotes the symmetric square root of $\PrCov$; see \cite[p. 413]{Flath2011}.
If $\fwdOper$ is replaced by an approximation $\hatFwdOper\in\R^{m\times d}$, then the resulting approximations of the posterior mean and covariance are given by
\begin{subequations}
\label{eq:approx_posterior}
\begin{align}
    \hatPosMn(\data) &= \PrMn+\PrCov \widehat{\fwdOper}^\top(\widehat{\fwdOper}\PrCov\widehat{\fwdOper}^\top+\ObsCov)^{-1}(\data-\widehat{\fwdOper}\PrMn)
    \label{eq:pos_mean_approx}
    \\
    \hatPosCov &= \PrCov - \PrCov \widehat{\fwdOper}^\top(\widehat{\fwdOper}\PrCov \widehat{\fwdOper}^\top+\ObsCov)^{-1}\widehat{\fwdOper}\PrCov.
    \label{eq:pos_covariance_approx}
\end{align}
\end{subequations}
The corresponding approximate Hessian and its prior-preconditioned version are given by $\widehat{\fwdOper}^\top \ObsCov^{-1}\widehat{\fwdOper}$ and $\PrCov^{1/2}\widehat{\fwdOper}^\top \ObsCov^{-1}\widehat{\fwdOper}\PrCov^{1/2}$.

\subsection{Optimal low-rank approximation}
\label{subsection:optimal_low_rank_approximation}

We recall some key results about the optimal low-rank approximations of linear Gaussian inverse problems.
Let $\rootPrCov\in\R^{d\times n}$ and $\rootObsCov\in\R^{m\times m}$ be possibly nonsymmetric matrices such that $\PrCov=\rootPrCov\rootPrCov^\top$ and $\ObsCov=\rootObsCov\rootObsCov^\top$.
While $\rootPrCov$ may be a rectangular matrix if $d\neq n$, $\rootObsCov$ must be a square matrix.
The dimension $n$ corresponds to $\rank{\PrCov}$, which satisfies $n\leq d$ since $\PrCov$ is the covariance matrix of a $\R^d$-valued random variable.

We now consider the prior-preconditioned Hessians $\rootPrCov^\top\fwdOper^\top \ObsCov^{-1}\fwdOper\rootPrCov$ and $\rootPrCov^\top \widehat{\fwdOper}^\top \ObsCov^{-1}\widehat{\fwdOper}\rootPrCov$, where we note that we use the possibly nonsymmetric square root $\rootPrCov$ instead of the symmetric square root $\PrCov^{1/2}$.
Now consider the singular value decomposition (SVD) of the root $\rootObsCov^{-1} \fwdOper \rootPrCov$ of the prior-preconditioned Hessian:
\begin{equation}
\label{eq:PPH_SVD}
\begin{aligned}
    &\rootObsCov^{-1} \fwdOper \rootPrCov=\sum_{i=1}^{m\wedge n} \SgVal_i \lSgVec_i \rSgVec_i ^\top,\qquad \rankPPH\coloneqq \rank{\rootObsCov^{-1} \fwdOper \rootPrCov},
    \\
    &\lSgMat\coloneqq[\lSgVec_1|\cdots|\lSgVec_m]\in\R^{m\times m},\quad  \SgValMat\coloneqq \text{diag}(\SgVal_i)_{i=1}\in\R^{m\times n},\quad \rSgMat\coloneqq [\rSgVec_1|\cdots|\rSgVec_d]\in\R^{n\times n},
    \end{aligned}
\end{equation}
where the singular values are in decreasing order, i.e. $\SgVal_i\geq \SgVal_{i+1}$ for every $i$, and where $\lSgMat$ and $\rSgMat$ are orthogonal matrices.
Note that $\rankPPH\leq  m\wedge n$.

In the following result, we recall some key properties of optimal low-rank approximations from the literature. Below,
\begin{equation*}
 \KLD(\mu\Vert \nu)\coloneqq \begin{cases}
                              \int \log\frac{\mathrm{d}\mu}{\mathrm{d}\nu}\mathrm{d}\mu, & \mu \ll\nu,
                              \\
                              +\infty, & \text{otherwise},
                             \end{cases}
\end{equation*}
denotes the Kullback-Leibler divergence of $\mu$ with respect to $\nu$, for any two probability measures $\mu$ and $\nu$ defined on a common measurable space.

\begin{proposition}[Optimal low-rank approximations for invertible prior covariances]
\label{proposition:optimal_low_rank_approximation}
Suppose that \Crefass{assumption_main} holds, $\PrMn=\mbf{0}$, $\rootPrCov\in\R^{d\times d}$, $d=n$ in \eqref{eq:PPH_SVD}, and that $\PrCov$ is invertible.
Define
\begin{subequations}
\begin{align}
   \label{eq:hat_eigenvectors_and_tilde_eigenvectors}
 \widehat{\bm{w}}_i\coloneqq & \rootPrCov \rSgVec_i, & &\widetilde{\bm{w}}_i\coloneqq \PrCov^{-1} \widehat{\bm{w}}_i, & & 1\leq i\leq d,
 \\
 \label{eq:OLR_approximate_forward_operator_and_optimal_projection}
\hatFwdOper^\dagger(r)\coloneqq &\fwdOper\mbf{P}^\dagger_r,& & \mbf{P}^\dagger_r\coloneqq \PrCov^{1/2} \rSgMat_r  (\PrCov^{-1/2}\rSgMat_r)^\top,& & 1\leq r\leq \rankPPH.
\end{align}
\end{subequations}
\begin{enumerate}
 \item \label{item_variance_reduction}
For $1\leq i\leq d$, the variance of the posterior marginal along $\text{span}(\widetilde{\bm{w}}_i)$ divided by the variance of the prior marginal along $\text{span}(\widetilde{\bm{w}}_i)$ is exactly $(1+\SgVal_i^2)^{-1}$, for the singular value $\SgVal_i$ in \eqref{eq:PPH_SVD}.
\item \label{item_oblique_projector}
Let $\hatPosMn^\dagger(\data)$ and $\hatPosCov^\dagger$ be the corresponding posterior mean and posterior covariance obtained by replacing $\hatFwdOper$ in \eqref{eq:approx_posterior} with $\hatFwdOper^\dagger(r)$ from \eqref{eq:OLR_approximate_forward_operator_and_optimal_projection}.
Define
\begin{equation*}
 \mathscr{A}_r\coloneqq \{\mbf{A}\in\R^{d\times m}:\rank{\mbf{A}}\leq r\},\quad \mathscr{C}_r\coloneqq\{\mbf{C}=\PrCov-\mbf{K}\mbf{K}^\top,\rank{\mbf{K}}\leq r \}.
\end{equation*}
Then there exists $\mbf{A}^\dagger\in\mathscr{A}_r$ such that $\hatPosMn^\dagger(\data)=\mbf{A}^\dagger \data$ for every $\data\in\R^m$, and
\begin{equation*}
(\mbf{A}^\dagger,\hatPosCov^\dagger)=\argmin_{\mbf{A}\in\mathscr{A}_r,\mbf{C}\in\mathscr{C}_r} \mathbb{E}[\KLD(\mathcal{N}(\mbf{A}\dataRV,\mbf{C})\Vert \mathcal{N}(\PosMn(\data),\PosCov))].
\end{equation*}
\end{enumerate}
\end{proposition}
\begin{proof}
~The definitions in \eqref{eq:hat_eigenvectors_and_tilde_eigenvectors} follow from Remark 4 and equation (3.2) in \cite{Spantini2015}.

\Cref{item_variance_reduction}: See \cite[p. A2458]{Spantini2015} for the first statement of this result and \cite[Proposition 3.6]{CarereLie2025b} for a more recent formulation.

 \Cref{item_oblique_projector}: This follows from \cite[Proposition 7.1]{CarereLie2025b}.
\end{proof}

By \eqref{eq:PPH_SVD}, $\delta_i>0$ for $i\leq \rankPPH$ and $\delta_i=0$ for $i>\rankPPH$. Thus, \Crefprop{proposition:optimal_low_rank_approximation}\ref{item_variance_reduction} implies that for $i\leq \rankPPH$, the variance of the posterior marginal along $\text{span}(\tilde{\bm{w}}_i)$ is strictly less than the variance of the prior marginal along $\text{span}(\tilde{\bm{w}}_i)$, and that the marginal variances are equal for $i>\rankPPH$.
This implies that $\text{span}((\widetilde{\mbf{w}}_i)_{i=1}^{r})$ is the $r$-dimensional subspace on which the posterior covariance differs most from the prior covariance, and that the posterior covariance differs from the prior covariance only on a subspace of dimension $\rankPPH$.
The `likelihood-informed subspace' is given by $\PrCov(\text{span}((\widetilde{\mbf{w}}_i)_{i=1}^{r}))$; see the discussion after \cite[Proposition 3.6]{CarereLie2025b}.
Note that $\PrCov(\text{span}((\widetilde{\mbf{w}}_i)_{i=1}^{r}))=\text{span}((\widehat{\bm{w}}_i)_{i=1}^{r})$, by \eqref{eq:hat_eigenvectors_and_tilde_eigenvectors}.

\Crefprop{proposition:optimal_low_rank_approximation}\ref{item_oblique_projector} shows that $\hatFwdOper^\dagger(r)$ yields an optimal rank-$r$ approximation of $\fwdOper$, in the sense that a corresponding approximate Gaussian posterior yields the smallest average Kullback--Leibler divergence with respect to the exact Gaussian posterior, over a class $\mathscr{A}_r$ of linear rank-$r$ data-to-posterior mean maps and a class $\mathscr{C}_r$ of symmetric negative updates by rank-$r$ matrices of the prior covariance $\PrCov$. The matrix $\hatFwdOper^\dagger(r)$ is obtained by precomposing $\fwdOper$ with $\mbf{P}^\dagger_r$, which is a projector of at most rank $r$ that is in general oblique with respect to the standard inner product. The projector $\mbf{P}^\dagger_r$ can be understood as being an optimal rank-$r$ projector in the sense of minimising the average Kullback--Leibler divergence, as described above.
A related statement concerning the projector $\mbf{P}^\dagger$ is given in \cite[Corollary 3.2]{Spantini2015}, which states that $\mbf{P}^\dagger$ yields optimal approximations of \emph{only} the posterior covariance in the class $\mathscr{C}_r$.

We shall consider low-rank approximations in a more general context than that of \Crefprop{proposition:optimal_low_rank_approximation}, by requiring only that \Crefass{assumption_main} holds.
Our results thus apply also to singular prior covariances.
Our motivation for considering this more general context comes from inverse problems for static structural loads, where reasonable prior covariances are often rank-deficient because uncertainty does not occur in every degree of freedom of the system; see e.g. \cite[Section 3]{Scheffels2025} and \cite{konigDimensionModelReduction2025}.

Recall from \eqref{eq:PPH_SVD} that $\rSgVec_i\in\R^n$ for $i=1,\ldots,n$, and that $\rankPPH\leq n$.
For $\rootPrCov\in\R^{d\times n}$ satisfying $\PrCov=\rootPrCov\rootPrCov^\top$, define
\begin{equation}
    \label{eq:hat_eigenvectors_and_tilde_eigenvectors_general_prior_covariances}
 \widehat{\bm{w}}_i\coloneqq  \rootPrCov \rSgVec_i, \quad \widetilde{\bm{w}}_i\coloneqq \fwdOper^\top\rootObsCov^{-\top}\lSgVec_i\SgVal_i^{-1},\quad 1\leq i\leq \rankPPH.
\end{equation}
The differences between the definitions in \eqref{eq:hat_eigenvectors_and_tilde_eigenvectors} and the definitions in \eqref{eq:hat_eigenvectors_and_tilde_eigenvectors_general_prior_covariances} consist in the definitions of $\widetilde{\bm{w}}_i$ and the range of the index $i$.
The following result shows that the definitions in \eqref{eq:hat_eigenvectors_and_tilde_eigenvectors_general_prior_covariances} recover the definitions in \eqref{eq:hat_eigenvectors_and_tilde_eigenvectors} if $\PrCov$ is invertible.
Thus, for the common range of indices $1\leq r\leq \rankPPH$, the definitions \eqref{eq:hat_eigenvectors_and_tilde_eigenvectors_general_prior_covariances} are more generally applicable than the definitions in \eqref{eq:hat_eigenvectors_and_tilde_eigenvectors}.

\begin{lemma}
 \label{lemma:tilde_w_j_definitions_consistent}
 Suppose \Crefass{assumption_main} holds and that $\PrCov$ is invertible.
 Let $d=n$ in \eqref{eq:PPH_SVD}, let $\rootPrCov\in\R^{d\times d}$ satisfy $\PrCov=\rootPrCov\rootPrCov^\top$, and let $1\leq i\leq \rankPPH$.
 If $\widehat{\bm{w}}_i =\rootPrCov \rSgVec_i$, then $\PrCov^{-1} \widehat{\bm{w}}_i=\fwdOper^\top \rootObsCov^{-1}\lSgVec_i\SgVal_i^{-1}$.
\end{lemma}

For the proof of \Creflem{lemma:tilde_w_j_definitions_consistent}, see Appendix~\ref{subsection:proofs_OLR}.

Unless otherwise specified, we shall use \eqref{eq:hat_eigenvectors_and_tilde_eigenvectors_general_prior_covariances} for the definitions of $(\widehat{\bm{w}}_i)_{i\leq \rankPPH}$ and $(\widetilde{\bm{w}}_i)_{i\leq \rankPPH}$ below.
For $1\leq r\leq \rankPPH$, define
\begin{subequations}
\label{eq:testMat_and_trialMat}
 \begin{align}
    \testMat_r&\coloneqq \rootPrCov\rSgMat_r=[\widehat{\bm{w}}_1|\cdots|\widehat{\bm{w}}_r]\in \R^{d\times r}
    \label{eq:testMat}
    \\
    \trialMat_r&\coloneqq \fwdOper^\top\rootObsCov^{-\top}\lSgMat_r(\SgValMat_r^\top\SgValMat_r)^{-1/2}=[\widetilde{\bm{w}}_1|\cdots|\widetilde{\bm{w}}_r]\in \R^{d\times r}.
    \label{eq:trialMat}
\end{align}
\end{subequations}

\begin{remark}
\label{remark:definitions_of_testMat_and_trialMat}
The definition of $\widehat{\bm{w}}_i$ in \eqref{eq:hat_eigenvectors_and_tilde_eigenvectors_general_prior_covariances} is identical to the definition of $\bm{v}_i$ above \cite[eq. (2.20)]{Scheffels2025}.
The definition of $\widetilde{\bm{w}}_i$ in \eqref{eq:hat_eigenvectors_and_tilde_eigenvectors_general_prior_covariances} is identical to the definition of $\bm{w}_i$ above \cite[eq. (2.20)]{Scheffels2025}.
Thus, the definition of $\testMat_r$ in \eqref{eq:testMat} and $\trialMat_r$ in \eqref{eq:trialMat} are identical to that given in \cite[eq. (2.21)]{Scheffels2025} and \cite[eq. (2.22)]{Scheffels2025} respectively.
\end{remark}

In preparation for the next result, we define for every $ 1\leq p\leq \infty$
\begin{equation}
 \label{eq:sum_trailing_singular_values}
 t(p,r)\coloneqq \begin{cases}
 \norm{(\SgVal_i)_{i\geq r+1}}_p & 1\leq r<\rankPPH
 \\
 0 & r\geq \rankPPH,
 \end{cases}
\end{equation}
the $\ell_p$-norm of the vector of trailing singular values of $\rootObsCov^{-1}\fwdOper\rootPrCov$.
\begin{proposition}[Optimal low-rank approximations for general prior covariances]
 \label{proposition:OLR_approximate_forward_operator_general_prior_covariance}
 Suppose \Crefass{assumption_main} holds, and let $(\widehat{\bm{w}}_i)_{i\leq \rankPPH}$ and $(\widetilde{\bm{w}}_i)_{i\leq \rankPPH}$ be as in \eqref{eq:hat_eigenvectors_and_tilde_eigenvectors_general_prior_covariances}.
Let $1\leq p\leq\infty$.
For $1\leq r\leq \rankPPH$, let
\begin{equation}
 \label{eq:OLR_approximate_forward_operator_general_prior_covariance}
  \hatFwdOperOLR(r)\coloneqq \fwdOper \mbf{P}_r,\quad \mbf{P}_r\coloneqq \testMat_r\trialMat_r^\top.
 \end{equation}
Then
\begin{align}
\rootObsCov^{-1}\hatFwdOperOLR(r)\rootPrCov=&\rootObsCov^{-1}\fwdOper  \rootPrCov \rSgMat_r  \rSgMat_r^\top
 \label{eq:OLR_approximate_forward_operator_general_prior_covariance_PPH}
 \\
     \Norm{\rootObsCov^{-1}(\fwdOper-\hatFwdOperOLR(r))\rootPrCov}_p=& t(p,r).
     \label{eq:OLR_errror_in_PPH}
 \end{align}
If in addition $\PrMn\in\range \rootPrCov$, then the approximate posterior covariance $\hatPosCovOLR(r)$ and approximate posterior mean $\hatPosMnOLR$ obtained by replacing $\hatFwdOper$ in \eqref{eq:approx_posterior} with $\hatFwdOperOLR(r)$ satisfy, for any realisation $\data$ of the data random variable $\dataRV$ from \eqref{eq:observation_model},
\begin{equation}
\label{eq:OLR_POS_Bound}
    \Norm{\PosCov-\hatPosCovOLR}_p\le C_1t(p,r),\qquad
    \Norm{\PosMn(\data)-\hatPosMnOLR(\data)}_2\le C_2 t(\infty,r),
\end{equation}
where $C_1=C_1(\rootPrCov,\rootObsCov,\fwdOper,\hatFwdOperOLR(r))$ and $C_2=C_2(C_1,\data,\PrMn,\rootPrCov,\rootObsCov,\fwdOper,\hatFwdOperOLR(r))$ are finite and do not depend on $p$.
Finally, if $\PrCov$ is invertible and $\rootPrCov\in\R^{d\times d}$ satisfies $\PrCov=\rootPrCov\rootPrCov^\top$, then $\testMat_r\trialMat_r^\top=\rootPrCov \rSgMat_r (\rootPrCov^{-\top} \rSgMat_r)^\top$.
\end{proposition}

For the proof of \Crefprop{proposition:OLR_approximate_forward_operator_general_prior_covariance}, see Appendix~\ref{subsection:proofs_OLR}.

The projector $\mbf{P}_r$ in \eqref{eq:OLR_approximate_forward_operator_general_prior_covariance} has been proposed in the literature on linear Gaussian inverse problems with rank-deficient prior covariances, where it is sometimes referred to as the `LIS projector'.
See e.g. \cite[Section 3.1.1]{konigDimensionModelReduction2025}, and note that the roles of $\testMat_r$ and $\trialMat_r$ here are reversed in \cite{konigDimensionModelReduction2025}.
In \cite[Theorem 3.2]{konigDimensionModelReduction2025}, the corresponding projected forward model $\fwdOper \mbf{P}_r$ is defined and shown to have the same optimality properties that were first shown in \cite{Spantini2015} for the setting of invertible prior covariances, i.e. it yields approximate posterior covariances that are optimal in the F\"{o}rstner metric and approximate posterior means that are optimal in a data-averaged sense.
In contrast, \Crefprop{proposition:OLR_approximate_forward_operator_general_prior_covariance} states an identity \eqref{eq:OLR_approximate_forward_operator_general_prior_covariance_PPH} for the root prior-preconditioned Hessian $\rootObsCov^{-1}\hatFwdOperOLR\rootPrCov$, uses this to calculate precisely the $p$-Schatten norm of the error of $\rootObsCov^{-1}\hatFwdOperOLR\rootPrCov$, states bounds on the $p$-Schatten norms of the approximate posterior covariance for arbitrary $1\leq p\leq\infty$, and states a bound on the Euclidean norm on the approximate posterior mean $\hatPosMn(\data)$ for any realisation $\data$ of the data.
Moreover, the proof of \Crefprop{proposition:OLR_approximate_forward_operator_general_prior_covariance} does not involve the restriction of the original inverse problem to a lower-dimensional subspace on which the restricted prior covariance is invertible, unlike \cite[Theorem 3.2]{konigDimensionModelReduction2025}.

In the context of optimal low-rank approximations, \Crefprop{proposition:OLR_approximate_forward_operator_general_prior_covariance} is important for multiple reasons.
By the Eckart-Young theorem, \eqref{eq:OLR_errror_in_PPH} implies that the projector $\mbf{P}_r$ in \eqref{eq:OLR_approximate_forward_operator_general_prior_covariance} is optimal in the sense of yielding the smallest possible error $\norm{\rootObsCov^{-1}\fwdOper\rootPrCov-\rootObsCov\fwdOper\mbf{Q}_r\rootPrCov}_p$ among all rank-$r$ projectors $\mbf{Q}_r$.
Second, the final statement of \Crefprop{proposition:OLR_approximate_forward_operator_general_prior_covariance} shows that the projector $\mbf{P}_r$ in \eqref{eq:OLR_approximate_forward_operator_general_prior_covariance} is a generalisation of the optimal projector $\mbf{P}_r^\dagger$ from \eqref{eq:OLR_approximate_forward_operator_and_optimal_projection} to the setting of singular prior covariances and to possibly non-symmetric square roots $\rootPrCov$ of $\PrCov$.
By recalling from the discussion of \Crefprop{proposition:optimal_low_rank_approximation}\ref{item_variance_reduction} that $\text{span}((\widetilde{\mbf{w}}_i)_{i=1}^{r})$ for $(\widetilde{\mbf{w}}_i)_{i\leq r}$ in \eqref{eq:hat_eigenvectors_and_tilde_eigenvectors} is the $r$-dimensional subspace on which the posterior covariance differs most from the prior covariance, we obtain the third reason why \Crefprop{proposition:OLR_approximate_forward_operator_general_prior_covariance} is important: it indicates that for s.p.s.d. $\PrCov$, $\range{\trialMat_r}$ plays the role of the $r$-dimensional subspace on which the largest relative variance reduction occurs.

\subsection{Reduced-order model}
\label{subsection:ROM}

To compute the exact posterior covariance and posterior mean using \eqref{eq:posterior}, we must evaluate matrix-vector products using the exact forward model $\fwdOper$.
However, since $\fwdOper=\ObsOper \fullModel^{-1}$ by \eqref{eq:forward_operator}, evaluating $\fwdOper$ involves solving a linear system of equations with the matrix $\fullModel\in\R^{d\times d}$, which often becomes increasingly expensive if $d\gg 1$. This motivates the use of a reduced-order model (ROM) of the exact forward model $\fwdOper$.
More precisely, we will replace the reference solution operator $\fullModel^{-1}$ with a ROM of $\fullModel^{-1}$, while leaving the observation operator $\ObsOper$ unchanged.
We shall focus on the specific ROM from \cite[Section 3]{Scheffels2025}, whose definition we recall below.
Unlike \Crefprop{proposition:optimal_low_rank_approximation}, we do not assume that $\PrCov$ is invertible, or that $\PrMn=\mbf{0}$; we only require that \Crefass{assumption_main} holds.

Recall the SVD $\rootObsCov^{-1} \fwdOper \rootPrCov=\lSgMat\SgValMat\rSgMat^\top$ in \eqref{eq:PPH_SVD}.
For $1\leq r\leq \rankPPH=\rank{\SgValMat}$, we define a rank-$r$ ROM to approximate the full-order model $\fullModel$, using Petrov-Galerkin projection.
Recall from the column submatrix notation in \Cref{subsection:notation} that $\lSgMat_r$ and $\rSgMat_r$ denote the matrices of the $r$ leading left and right singular vectors of $\rootObsCov^{-1}\fwdOper \rootPrCov$ respectively.
For the Petrov-Galerkin projection, we shall set the test and trial bases to be $\testMat_r$ and $\trialMat_r$ respectively, for $\testMat_r$ and $\trialMat_r$ given in \eqref{eq:testMat_and_trialMat}.
This leads to the following definition from \cite[Section 3]{Scheffels2025} of the rank-$r$ ROM of the full-order model $\fullModel$ in \eqref{eq:observation_model}:
\begin{equation}
\label{eq:ROM_of_PDE}
 \redModel(r)\coloneqq \trialMat_r^\top \fullModel\testMat_r\in\R^{r\times r}.
\end{equation}
The advantage of using $\redModel(r)$ instead of $\fullModel$ is that it is computationally cheaper to invert $\redModel(r)$ instead of inverting $\fullModel$, whenever $r<d$.
We now make the following assumption about $\redModel(r)$.

\begin{assumption}
 \label{assumption:ROM_of_PDE_invertible}
 {\rm For every $1\leq r\leq \rankPPH$, $\redModel(r)$ is invertible.}
\end{assumption}

Given \Crefass{assumption:ROM_of_PDE_invertible}, we can define a ROM-based approximate forward model:
\begin{equation}
\label{eq:approximate_forward_operator}
     \hatFwdOperROM(r)\coloneqq \ObsOper \testMat_r\redModel(r)^{-1}\trialMat_r^\top\in\R^{m\times d}.
\end{equation}
Recall from \Crefprop{proposition:optimal_low_rank_approximation}\ref{item_oblique_projector} that in the case of invertible prior covariances, the projector $\mbf{P}^\dagger_r$ defined in \eqref{eq:OLR_approximate_forward_operator_and_optimal_projection} yields an approximate forward model $\hatFwdOper^\dagger=\fwdOper\mbf{P}^\dagger_r$ that is optimal in the sense of minimising the average Kullback--Leibler divergence.
Recall from the final statement of \Crefprop{proposition:OLR_approximate_forward_operator_general_prior_covariance} that if $\PrCov$ is invertible and if we replace $\rootPrCov$ with $\PrCov^{1/2}$ in \eqref{eq:testMat}, then $\mbf{P}^\dagger_r=\testMat_r \trialMat_r^\top$.
However, even under these additional conditions, we do not expect that $\hatFwdOperROM=\hatFwdOperOLR$ will hold in general, because the first equation below does not hold in general:
\begin{equation*}
 \ObsOper \testMat_r \redModel(r)^{-1}\trialMat_r^\top=\ObsOper \fullModel^{-1}\testMat_r\trialMat_r^\top= \ObsOper \fullModel^{-1} \mbf{P}_r=\hatFwdOperOLR(r).
\end{equation*}
Thus, we do not expect that $\hatFwdOperROM(r)$ in \eqref{eq:approximate_forward_operator} will yield either an optimal low-rank approximation of the exact posterior in the sense of \Crefprop{proposition:optimal_low_rank_approximation}\ref{item_oblique_projector}, or an optimal approximation of the root of the exact prior-preconditioned Hessian in the sense of \Crefprop{proposition:OLR_approximate_forward_operator_general_prior_covariance}.
In the next section, we address the problem of determining the $p$-Schatten norm error of the approximate root prior-preconditioned Hessian $\rootObsCov^{-1}\hatFwdOperROM(r)\rootPrCov$.

\section{Error analysis}
\label{section:error_analysis}

In this section, we analyse the error in the approximate root prior-preconditioned Hessian associated to the reduced forward model $\hatFwdOperROM(r)$ defined in \eqref{eq:approximate_forward_operator}.
We use this analysis to prove upper bounds on the $p$-Schatten norm error in the posterior covariance and posterior mean corresponding to $\hatFwdOper(r)$.

Recall the column submatrix notation from \Cref{subsection:notation}, and the notation from \eqref{eq:PPH_SVD}.
The first result of this section is a useful characterisation of $\hatFwdOperROM(r)$.

 \begin{lemma}
 \label{lemma:decomposition_of_root_of_approximate_PPH}
 Suppose \Crefass{assumption_main} and \Crefass{assumption:ROM_of_PDE_invertible} hold, and let $1\leq r\leq \rankPPH$. Then
 \begin{equation}
\hatFwdOperROM(r)=\ObsOper \testMat_r \left( \lSgMat_r^\top \rootObsCov^{-1} \ObsOper \testMat_r\right)^{-1} \lSgMat_r^\top\rootObsCov^{-1}\fwdOper
\label{eq:approximate_forward_operator_ROM}
\end{equation}
and in particular
\begin{equation}
  \lSgMat_r \lSgMat_r^\top \rootObsCov^{-1}\hatFwdOperROM(r)\rootPrCov= \lSgMat_r  \lSgMat_r^\top \rootObsCov^{-1}\fwdOper\rootPrCov.
  \label{eq:approximate_forward_operator_ROM_post_projection_property}
 \end{equation}
\end{lemma}

The proof of \Creflem{lemma:decomposition_of_root_of_approximate_PPH} is given in Appendix~\ref{subsection:proofs_error_analysis}.

The significance of \eqref{eq:approximate_forward_operator_ROM_post_projection_property} is that the approximate root prior-preconditioned Hessian $\rootObsCov^{-1}\hatFwdOperROM(r)\rootPrCov$ preserves the component of the exact root prior-preconditioned Hessian $\rootObsCov^{-1}\fwdOper\rootPrCov$ in the range of $\lSgMat_r$.
This property of $\hatFwdOperROM$ differs from the property \eqref{eq:OLR_approximate_forward_operator_general_prior_covariance_PPH} of $\hatFwdOperOLR$, thus providing an additional reason to expect that $\hatFwdOperROM$ in general differs from $\hatFwdOperOLR$.
By \eqref{eq:approximate_forward_operator_ROM_post_projection_property} , if $r=m$, then $\lSgMat_m=\lSgMat$ by the column submatrix notation in \Cref{subsection:notation}.
Thus $\lSgMat_m\lSgMat_m^\top =\Id_m$, which implies $\rootObsCov^{-1}\hatFwdOperROM(r)\rootPrCov=\rootObsCov^{-1}\fwdOper\rootPrCov$.
The property \eqref{eq:approximate_forward_operator_ROM_post_projection_property}  also facilitates an analysis of the error $\rootObsCov^{-1}\hatFwdOperROM(r)\rootPrCov-\rootObsCov^{-1}\fwdOper\rootPrCov$, which leads to the main theoretical result of this work, \Cref{theorem:error_decomposition_approximate_root_PPH}, below.
To this end, recall from \eqref{eq:sum_trailing_singular_values} that $t(p,r)$ is the $\ell_p$ norm of the vector of trailing singular values. We shall bound the $p$-Schatten norm of $\rootObsCov^{-1}\hatFwdOper(r)\rootPrCov-\rootObsCov^{-1}\fwdOper\rootPrCov$ by
 \begin{equation}
 \label{eq:bound_as_function_of_p}
 \begin{aligned}
 b(p,r)\coloneqq & \Norm{\lSgMat_\bot (\SgValMat_\bot^\top \SgValMat_\bot)^{1/2} \rSgMat_\bot^\top}_p
 \\
 &+
 \Norm{\lSgMat_\bot \lSgMat_\bot^\top\rootObsCov^{-1} \ObsOper \testMat_r (\lSgMat_r^\top \rootObsCov^{-1} \ObsOper \testMat_r)^{-1} (\SgValMat_r^\top\SgValMat_r)^{1/2}\rSgMat_r^\top}_p
 \\
 =& t(p,r)+
 \Norm{\lSgMat_\bot \lSgMat_\bot^\top\rootObsCov^{-1} \ObsOper \testMat_r (\lSgMat_r^\top \rootObsCov^{-1} \ObsOper \testMat_r)^{-1} (\SgValMat_r^\top\SgValMat_r)^{1/2}\rSgMat_r^\top}_p.
\end{aligned}
\end{equation}
\begin{remark}[Properties of $b(p,r)$]
\label{remark:properties_of_bound_as_function_of_p}
Let $1\leq p\leq \infty$. By \eqref{eq:sum_trailing_singular_values}, the first term $t(p,r)$ on the right-hand side of \eqref{eq:bound_as_function_of_p} vanishes if $r\geq \rankPPH$.
By the column submatrix notation in \Cref{subsection:notation} and the fact that $\lSgMat$ in \eqref{eq:PPH_SVD} is an $m\times m$ matrix, it follows that $\lSgMat_\bot=\mbf{0}$ if $r\geq m$, and thus the second term on the right-hand side of \eqref{eq:bound_as_function_of_p} vanishes.
Since $m\geq \rankPPH$ by \eqref{eq:PPH_SVD}, we conclude that if $r \geq m$, then $b(p,r)=0$.
\end{remark}

\begin{theorem}
 \label{theorem:error_decomposition_approximate_root_PPH}
 Suppose \Crefass{assumption_main} and \Crefass{assumption:ROM_of_PDE_invertible} hold, and let $1\leq r \leq \rankPPH$. Then
 \begin{subequations}
 \begin{align}
&\rootObsCov^{-1}(\fwdOper-\hatFwdOperROM(r))\rootPrCov
\nonumber\\
=& \lSgMat_\bot \lSgMat_\bot^\top \left( \Id_m- \rootObsCov^{-1}\ObsOper\testMat_r (\lSgMat_r^\top \rootObsCov^{-1}\ObsOper \testMat_r)^{-1} \lSgMat_r^\top\right)\rootObsCov^{-1} \fwdOper \rootPrCov
\label{eq:decomposition_of_error_of_root_PPH}
\\
=& \lSgMat_\bot (\SgValMat_\bot^\top \SgValMat_\bot)^{1/2} \rSgMat_\bot^\top
  -\lSgMat_\bot \lSgMat_\bot^\top\rootObsCov^{-1}\ObsOper \testMat_r \left( \lSgMat_r^\top \rootObsCov^{-1} \ObsOper \testMat_r\right)^{-1}(\SgValMat_r^\top\SgValMat_r)^{1/2}\rSgMat_r^\top.
  \label{eq:decomposition_of_error_of_root_PPH_alternate_form}
\end{align}
\end{subequations}
For $1\leq p\leq \infty$,
\begin{equation}
 \Norm{\rootObsCov^{-1}(\fwdOper-\hatFwdOperROM(r))\rootPrCov}_p\leq b(p,r).
 \label{eq:bound_on_error_of_root_PPH}
 \end{equation}
 Finally, if in addition $\PrMn\in\range \rootPrCov$, then the approximate posterior covariance $\hatPosCovROM$ and approximate posterior mean $\hatPosMnROM$ obtained by replacing $\hatFwdOper$ in \eqref{eq:approx_posterior} with $\hatFwdOperROM(r)$ satisfy, for any realisation $\data$ of the data random variable $\dataRV$ from \eqref{eq:observation_model},
  \begin{equation}
  \Norm{\PosCov-\hatPosCovROM}_p
  \leq C_1^\prime b(p,r),\qquad \Norm{\PosMn(\data)-\hatPosMnROM(\data)}_2
  \leq  C_2^\prime b(\infty,r),
  \label{eq:pos_error_ROM}
 \end{equation}
 where $C_1^\prime =C_1^\prime(\rootPrCov,\rootObsCov,\fwdOper,\hatFwdOperROM(r))$ and $C_2^\prime =C_2^\prime (C_1^\prime,\data,\PrMn,\rootPrCov,\rootObsCov,\fwdOper,\hatFwdOperROM(r))$ are finite and do not depend on $p$.
\end{theorem}

The proof of \Cref{theorem:error_decomposition_approximate_root_PPH} is given in Appendix~\ref{subsection:proofs_error_analysis}.

 By the properties of $b(p,r)$ described in \Crefrem{remark:properties_of_bound_as_function_of_p}, we conclude that if $r\geq m$, then $b(p,r)=0$ for any value of $p$. This is consistent with the observation made below \Creflem{lemma:decomposition_of_root_of_approximate_PPH}, namely that if $r=m$, then by \eqref{eq:approximate_forward_operator_ROM_post_projection_property} it holds that $\rootObsCov^{-1}\hatFwdOperROM(r)\rootPrCov=\rootObsCov^{-1}\fwdOper\rootPrCov$. In particular, \Cref{theorem:error_decomposition_approximate_root_PPH} also yields that if $r=m$, the error in the ROM-based approximation of the root prior-preconditioned Hessian will vanish, and the Gaussian approximate posterior obtained by replacing $\hatFwdOper$ with $\hatFwdOperROM$ in \eqref{eq:approx_posterior} agrees with the exact posterior in \eqref{eq:posterior}.

 \begin{remark}[Tightness of error bounds for $\hatPosCovROM$ and $\hatPosMnROM$]
 \label{remark:tightness_of_error_bounds}
The bounds in \eqref{eq:OLR_POS_Bound} and \eqref{eq:pos_error_ROM} are obtained using certain \textit{a priori} bounds in \cite[Theorem 3.1]{KoenigLie2025}.
These \textit{a priori} bounds are not in general tight for $1\leq r<\rankPPH$, because they follow from applications of the triangle inequality and the multiplicativity property of $p$-Schatten norms.
Thus, we do not expect the bounds in \eqref{eq:OLR_POS_Bound} and \eqref{eq:pos_error_ROM} to be tight.
\end{remark}

Recall that the bound \eqref{eq:OLR_errror_in_PPH} on the error $\norm{\rootObsCov^{-1} (\fwdOper-\hatFwdOperOLR(r))\rootPrCov}_p$ is in fact an equality for all $1\leq p\leq \infty$ and $1\leq r\leq \rankPPH$.
It is natural to ask whether the analogous bound in \eqref{eq:bound_on_error_of_root_PPH} for $\norm{\rootObsCov^{-1} (\fwdOper-\hatFwdOperROM(r))\rootPrCov}_p$ is also an equality for certain values of $p$.
Since \eqref{eq:bound_on_error_of_root_PPH} arises from an application of the triangle inequality with the $p$-Schatten norm to \eqref{eq:decomposition_of_error_of_root_PPH_alternate_form}, equality holds if and only if equality holds in the application of the triangle inequality.
  It is known that for every $s,t\in\N$, the Banach space $(\R^{s\times t},\norm{\cdot}_p)$ is strictly convex if and only if $1<p<\infty$; see \cite[Corollary 1]{Aziznejad26042021}.
 By the properties of strictly convex spaces, equality holds in the triangle inequality $\norm{\mbf{A}+\mbf{B}}_p\leq \norm{\mbf{A}}_p+\norm{\mbf{B}}_p$ if and only if $\mbf{A},\mbf{B}\in\R^{s\times t}$ satisfy the collinearity condition that $\mbf{A}=c \mbf{B}$ for some $c>0$.
 For the case $p=\infty$, equality holds in $\norm{\mbf{A}+\mbf{B}}_\infty\leq \norm{\mbf{A}}_\infty+\norm{\mbf{B}}_\infty$ if and only if there exist unit vectors $\mbf{v}^{(1)}\in\R^{s}$ and $\mbf{v}^{(2)}\in\R^{t}$ such that $\mbf{A}\mbf{v}^{(2)}=\norm{\mbf{A}}_\infty \mbf{v}^{(1)}$ and $\mbf{B}\mbf{v}^{(2)}=\norm{\mbf{B}}_\infty \mbf{v}^{(1)}$; see e.g. \cite{4709579} for a proof of necessity.
Now note that the first term $\lSgMat_\bot (\SgValMat_\bot^\top \SgValMat_\bot)^{1/2} \rSgMat_\bot^\top$ on the right-hand side of \eqref{eq:decomposition_of_error_of_root_PPH_alternate_form} satisfies $ \lSgMat_\bot (\SgValMat_\bot^\top \SgValMat_\bot)^{1/2} \rSgMat_\bot^\top\mbf{v}^{(2)}\neq \mbf{0}$ if and only if $\mbf{v}^{(2)}\in \range \rSgMat_\bot$.
 On the other hand, if $\mbf{v}^{(2)}\in\range \rSgMat_\bot$, then $\rSgMat_r \mbf{v}^{(2)}=\mbf{0}$.
 Thus, the first and second terms on the right-hand side of \eqref{eq:decomposition_of_error_of_root_PPH_alternate_form} do not satisfy either the collinearity condition for the case $1<p<\infty$, or the necessary condition for the $p=\infty$ case.
 This implies that if both terms on the right-hand side of \eqref{eq:decomposition_of_error_of_root_PPH_alternate_form} are nonzero, then equality in \eqref{eq:bound_on_error_of_root_PPH} does not hold for any $1<p\leq \infty$.

\section{Numerical experiments}
\label{section:numerical_experiments}
In this section, we present numerical experiments to compare the performance of approximate posterior measures obtained using the ROM-based approximation described in \Cref{subsection:ROM} against the performance of approximate posterior measures obtained using the optimal low-rank approximation described in \Cref{subsection:optimal_low_rank_approximation}.
We compare the error of the approximate root prior-preconditioned Hessian $\rootObsCov^{-1}\hatFwdOper^\bullet\rootPrCov$ for $\bullet=\text{OLR}$ and $\bullet=\text{ROM}$ to the corresponding bounds given in \eqref{eq:OLR_errror_in_PPH} and \eqref{eq:bound_on_error_of_root_PPH}.
We also compare the error of $\hatPosCov^\bullet$ and $\hatPosMn^\bullet$ for $\bullet=\text{OLR}$ and $\bullet=\text{ROM}$ to their corresponding bounds in \eqref{eq:OLR_POS_Bound} and \eqref{eq:pos_error_ROM}.
We consider two different test problems from structural engineering that involve the inference of static structural loads, namely a cantilever bar in \Cref{sec:Bar} and a two-dimensional plate in \Cref{sec:Plate}. The codes used to generate these results are available on GitHub\footnote{\url{https://github.com/jakobscheffels/Error_bounds_for_approximate_posteriors_from_LIS_ROMs.git}}.

Below, the parameters $d$, $m$, $n$, and $\rankPPH$ are as in \Crefass{assumption_main} and \eqref{eq:PPH_SVD}.
In both the examples below, the dimension $m$ of the data is smaller than the dimension $n$ of the uncertain component of the parameter, where $n=\rank{\PrCov}$.

\subsection{Cantilever bar}\label{sec:Bar}
We first consider a cantilever bar from \cite[Section 4.1]{Scheffels2025} governed by the static equilibrium equation and boundary conditions
\begin{flalign*}
    D\frac{d^2u(z)}{dz^2}+q(z)=0&,\\
    u(0)=0&,\\
    \frac{du(z)}{dz}\vert_{z=L}=0&,
\end{flalign*}
for $z\in[0,L]$ with $L=2~\text{m}$ and $D=4\times10^8~\text{N}$ as the constant rigidity, where $u(z)$ is the axial displacement and $q(z)$ is the distributed load modeled as a Gaussian random field. We assume a constant prior mean $\mu_Q=4\times 10^6 ~\text{N}$ and an exponential auto-correlation given by $\Gamma_{QQ}(z_1,z_2)=\sigma_Q^2\times\text{exp}(-\frac{\vert z_1-z_2\vert}{\theta})$ with standard deviation $\sigma_Q=0.3\times \mu_Q$ and correlation length $\theta=L/2$.

The governing equation is solved by the linear finite element method using $n=100$ linear bar elements with shape functions $N_1(\xi)=1-\frac{\xi\cdot n}{L}$ and $N_2(\xi)=\frac{\xi\cdot n}{L}$ with elemental coordinate $\xi\in[0,\frac{L}{n}]$.
We discretise the random field using the midpoint method \cite{der1988stochastic} and obtain the load vector $\mbf{q}\in\R^n$ with $\mbf{q}_i= q(\Bar{z}_i)$ for $1\le i\le n$, where $\Bar{z}_i$ is the midpoint of element $i$. The discretised approximation of the random load has a constant mean vector $\bm\mu_Q \coloneqq \mu_Q\times\mbf{1}_n\in \R^{n}$, where $\mbf{1}_n\in\R^{n}$ is a vector with all entries equal to $1$.
The covariance matrix $\mbf{\Gamma}_{QQ} \in \R^{n \times n}$ is calculated by $\bm\Gamma_{QQ,i,j}\coloneqq\Gamma_{QQ}(\Bar{z}_i,\Bar{z}_j)$ for all elements $1\le i,j\le n$.
By the choice of the exponential auto-correlation function, $\bm\Gamma_{QQ}$ is invertible.
Our choice of numerical discretisation yields $d=n$ in this example.
The nodal force vector $\unknown\in\R^d$ obtained through integration of the discretised random load over the linear finite element shape functions is given as $\unknown=\mbf{L}\mbf{q}$, where $\mbf{L}\in\R^{d\times n}$ maps the discretised load to the nodal force vector by integration of the element shape functions assembled such that the contribution of element $i$ to the $e$-th component of $\unknown$ is $\mbf{L}_{e,i} \coloneqq \int_{\Bar{z}_i-\frac{L}{2n}}^{\Bar{z}_i+\frac{L}{2n}} N_e(\xi)\, \mathrm{d} \xi$.
The unknown nodal force vector $\unknown$ therefore is distributed according to a Gaussian prior $\mathcal{N}(\PrMn,\PrCov)$ with $\PrMn\coloneqq \mbf{L}\bm\mu_Q\in\R^d$ and $\PrCov\coloneqq \mbf{L}\mbf{\Gamma}_{QQ}\mbf{L}^\top\in\R^{d\times d}$.
For this construction, $\PrCov$ is invertible since $\mbf{L}$ and $\mbf{\Gamma}_{QQ}$ are invertible.
Since $d=n$ in this case, it holds that $\range \rootPrCov=\R^n$ for any $\rootPrCov$ that satisfies $\PrCov=\rootPrCov\rootPrCov^\top$.
Hence the range condition $\PrMn\in\range\rootPrCov$ in \Crefprop{proposition:OLR_approximate_forward_operator_general_prior_covariance} and \Cref{theorem:error_decomposition_approximate_root_PPH} is satisfied.

To define the observation operator $\ObsOper$ in \Crefass{assumption_main}, we drew $m=10$ locations from the uniform distribution along the bar and fixed these as the locations at which we observe the state $\state$.
We defined the observation noise covariance $\ObsCov\coloneqq \sigma_{\text{obs}}^2\times\Id_m\in\R^{m\times m}$,
where $\Id_m$ is the $m$-dimensional identity matrix and $\sigma_{\text{obs}}=0.001$ is set to $5\%$ of the maximum mean displacement.
We generate the data $\data$ by sampling the random variable $\dataRV$ defined according to \Crefass{assumption_main}.
For this choice of $\ObsCov$ and for the choices described in the previous paragraph, we have $\rankPPH\coloneqq \rank{\rootObsCov^{-1}\mbf{G}\rootPrCov}= m$.
Thus, $\rootObsCov^{-1}\mbf{G}\rootPrCov$ has exactly $\rankPPH=m=10$ nonzero singular values $(\SgVal_i)_i$.

\begin{figure}
    \centering

    \begin{minipage}[b]{0.45\textwidth}
        \centering
        \includegraphics[width=\textwidth]{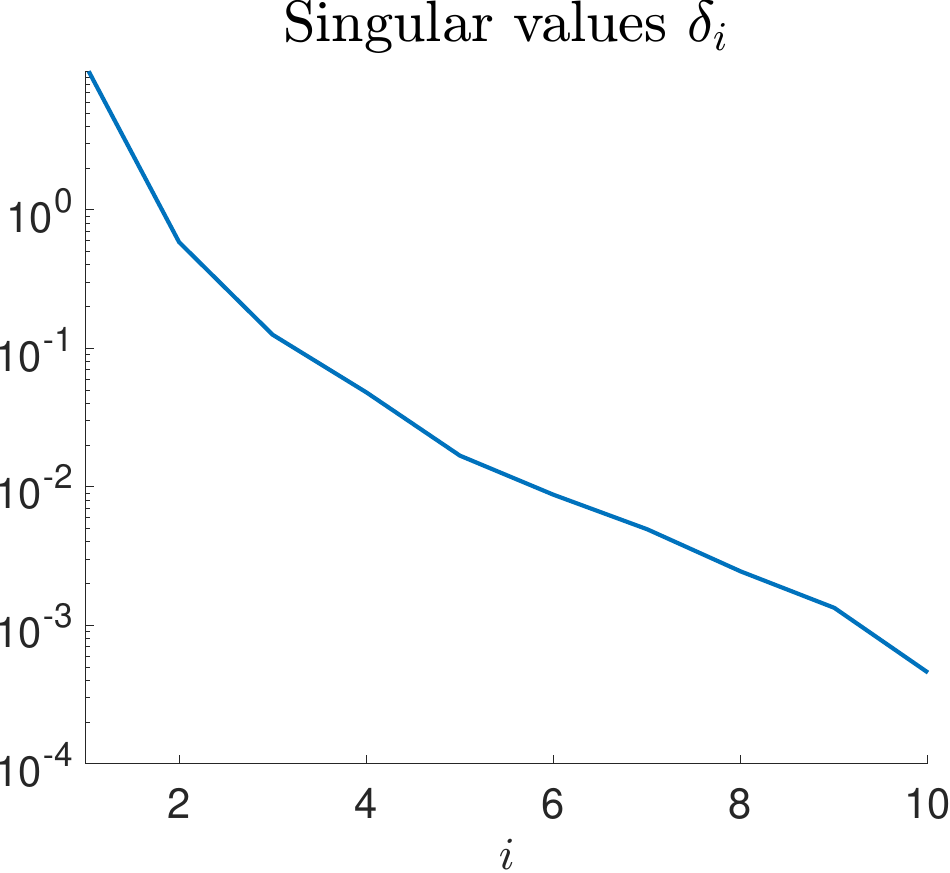}
        \small (a)
    \end{minipage}
    \hfill
    \begin{minipage}[b]{0.45\textwidth}
        \centering
        \includegraphics[width=\textwidth]{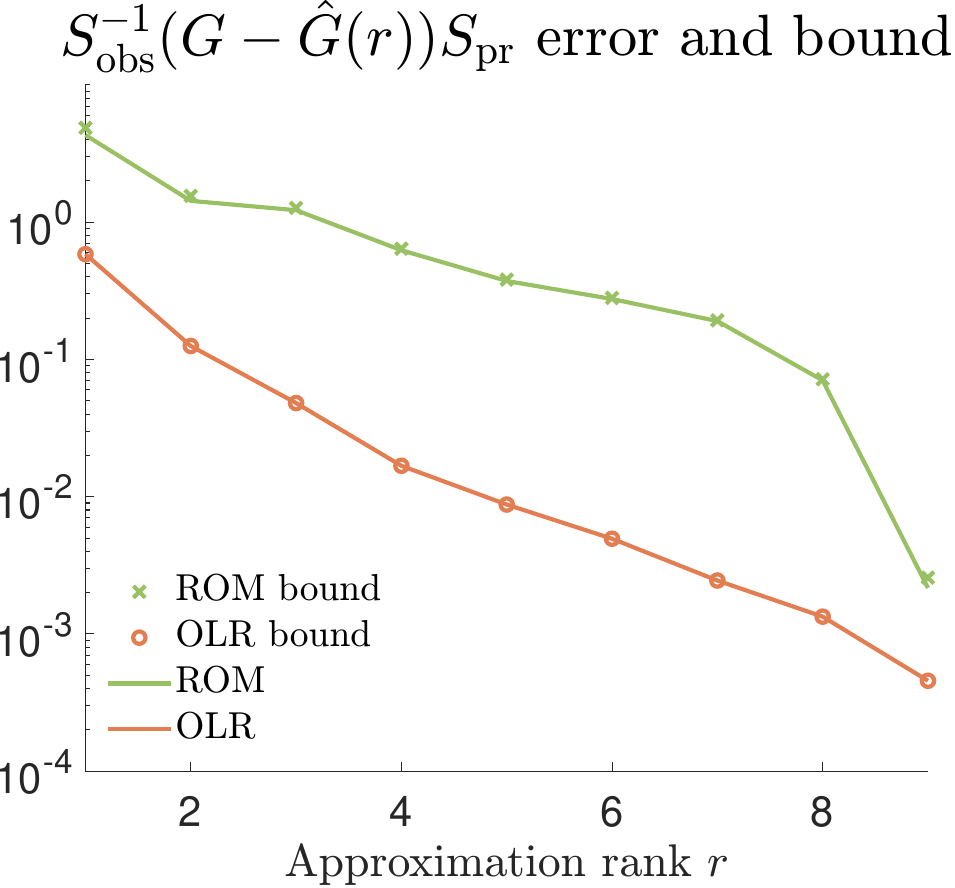}
        \small (b)
    \end{minipage}

    \medskip

    \begin{minipage}[b]{0.45\textwidth}
        \centering
        \includegraphics[width=\textwidth]{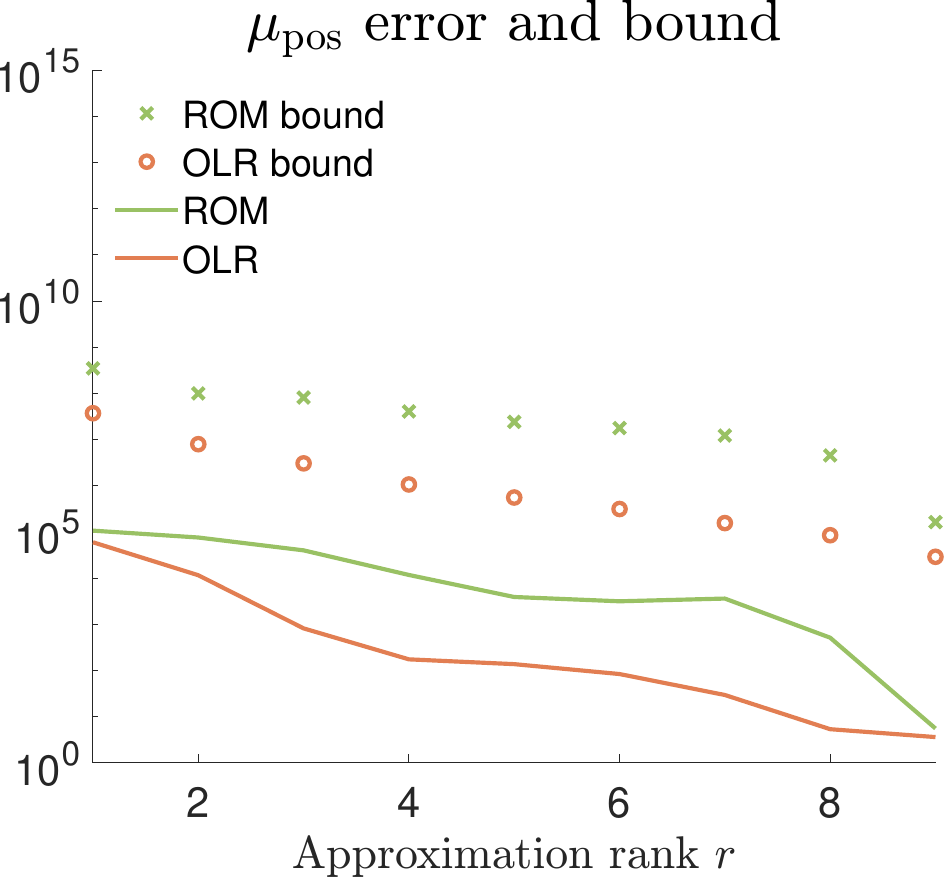}
        \small (c)
    \end{minipage}
    \hfill
    \begin{minipage}[b]{0.45\textwidth}
        \centering
        \includegraphics[width=\textwidth]{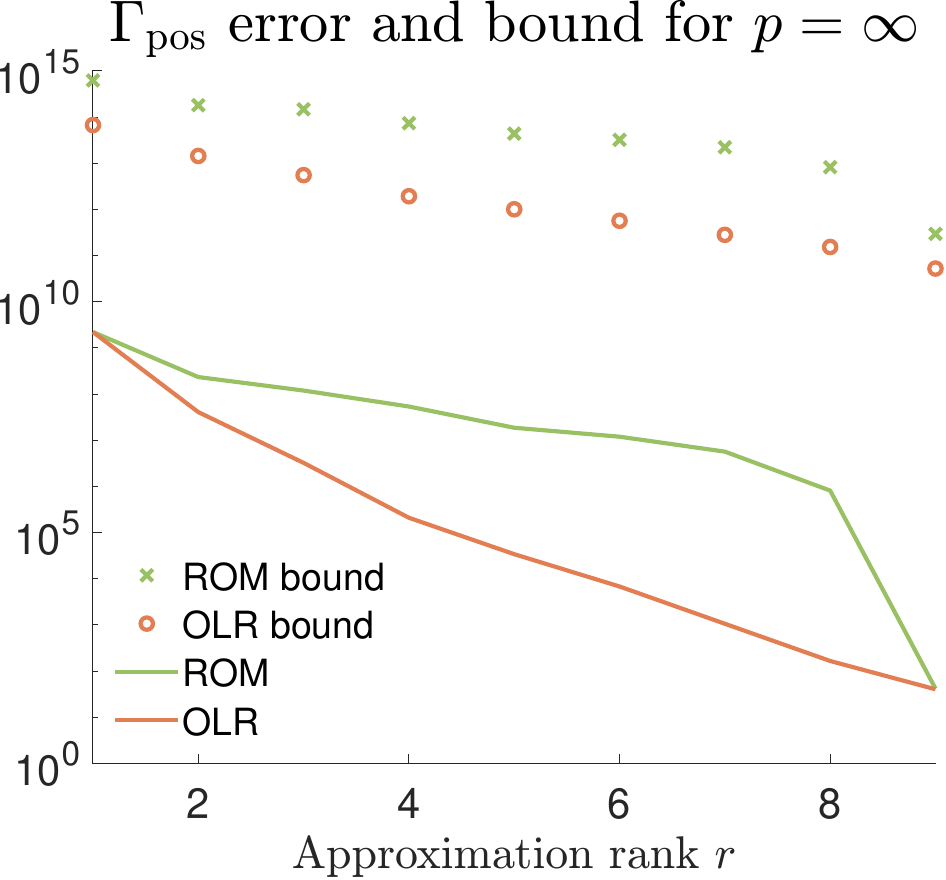}
        \small (d)
    \end{minipage}

    \caption{Comparison of $\rootObsCov^{-1}(\fwdOper-\hatFwdOper(r))\rootPrCov$, for $\hatFwdOper=\hatFwdOperOLR$ from \eqref{eq:OLR_approximate_forward_operator_general_prior_covariance} and $\hatFwdOper=\hatFwdOperROM$ from \eqref{eq:approximate_forward_operator}, for the cantilever bar problem in \Cref{sec:Bar}.
    \\
    (a): Singular values $(\SgVal_i)_i$ of $\rootObsCov^{-1}\fwdOper\rootPrCov$ from \eqref{eq:PPH_SVD}.
    \\
    (b): Errors $\Vert\rootObsCov^{-1}(\fwdOper-\hatFwdOper(r))\rootPrCov\Vert_\infty$ for $\hatFwdOper=\hatFwdOperOLR$ and $\hatFwdOper=\hatFwdOperROM$, the OLR bound $t(\infty,r)$ from \eqref{eq:OLR_errror_in_PPH}, and the ROM bound $b(\infty,r)$ from \eqref{eq:bound_on_error_of_root_PPH}.
    \\
       (c): Errors $\Vert \PosMn(\data)-\hatPosMn(\data)\Vert_2$ resulting from using $\hatFwdOperOLR$ and $\hatFwdOperROM$ in \eqref{eq:pos_mean_approx}, the OLR bound $C_2 t(\infty,r)$ from \eqref{eq:OLR_POS_Bound}, and the ROM bound $C_2^\prime b(\infty,r)$ from \eqref{eq:pos_error_ROM}.
    \\
    (d): Errors $\Vert \PosCov-\hatPosCov\Vert_\infty$ resulting from using $\hatFwdOperOLR$ and $\hatFwdOperROM$ in \eqref{eq:pos_covariance_approx}, the OLR bound $C_1 t(\infty,r)$ from \eqref{eq:OLR_POS_Bound}, and the ROM bound $C_1^\prime b(\infty,r)$ from \eqref{eq:pos_error_ROM}.}
    \label{fig:Bar}
\end{figure}

\Cref{fig:Bar}(a) shows the nonzero singular values $(\SgVal_i)_{i\leq \rankPPH}$ of $\rootObsCov^{-1}\fwdOper\rootPrCov$.
\Cref{fig:Bar}(b) shows the errors $\norm{\rootObsCov^{-1}(\fwdOper-\hatFwdOper(r))\rootPrCov}_\infty$ for $\hatFwdOper=\hatFwdOperOLR$ and $\hatFwdOper=\hatFwdOperROM$, the OLR error bound $t(\infty,r)$ from \eqref{eq:OLR_errror_in_PPH}, and the ROM error bound $b(\infty,r)$ from \eqref{eq:bound_on_error_of_root_PPH}.
By \eqref{eq:OLR_errror_in_PPH} and \eqref{eq:sum_trailing_singular_values} (respectively, \eqref{eq:bound_on_error_of_root_PPH} and \Crefrem{remark:properties_of_bound_as_function_of_p}), the OLR (resp. ROM) approximation recovers the exact root prior-preconditioned Hessian and its error vanishes if $r=\rankPPH$ (resp. $r=m$).
Since in this section $\rankPPH=m$, we only show the errors of the approximations for $r<\rankPPH$.

Since the $\infty$-Schatten norm is used to quantify the errors, we have
\begin{equation}
\label{eq:infty_Schatten_norm_of_error_equals_largest_trailing_singular_value}
 \norm{\rootObsCov^{-1}(\fwdOper-\hatFwdOperOLR(r))\rootPrCov}_\infty\overset{\text{\eqref{eq:OLR_errror_in_PPH}}}{=}t(\infty,r)\overset{\text{\eqref{eq:sum_trailing_singular_values}}}{=}\SgVal_{r+1},\quad 1\leq r< \rankPPH.
\end{equation}
This explains why the errors $\norm{\rootObsCov^{-1}(\fwdOper-\hatFwdOperOLR(r))\rootPrCov}_\infty$ in \Cref{fig:Bar}(b) match the values of $\SgVal_i$ in \Cref{fig:Bar}(a).
In \Cref{fig:Bar}(b), the observed errors $\norm{\rootObsCov^{-1}(\fwdOper-\hatFwdOperROM(r))\rootPrCov}_\infty$ for the ROM approximation are larger than the observed errors for the OLR approximation.
This is consistent with the expected behaviour that we described below \eqref{eq:approximate_forward_operator}.
Furthermore, the upper bound given in \eqref{eq:bound_on_error_of_root_PPH} for the ROM approximation is very close but not exactly equal to the observed error of the ROM approximation, due to the use of triangle inequality when deriving the bound.
On the other hand, equality between the bound $t(p,r)$ and the observed error of the OLR approximation is observed; this is consistent with \eqref{eq:OLR_errror_in_PPH}.
Finally, \Cref{fig:Bar}(c) and (d) show the errors $\norm{\PosMn(\data)-\hatPosMn(\data)}_2$ and $\norm{\PosCov-\hatPosCov}_\infty$ in the posterior mean and posterior covariance respectively for the ROM and OLR approximations, as well as the bounds on these errors given by \eqref{eq:OLR_POS_Bound} and \eqref{eq:pos_error_ROM}.
We observe that the bounds are not tight; this is due to the large values of the constants $(C_i,C_i^\prime)$ for $i=1,2$, and is consistent with our expectations in \Crefrem{remark:tightness_of_error_bounds}.
However, the decay in the bounds closely tracks the decay in the errors, thus providing numerical evidence for the idea that the error in the approximate root prior-preconditioned Hessian is the key determinant for the error in the approximate posterior mean and approximate posterior covariance. Recall that the errors of the posterior mean and covariance approximation are equal to $0$ if $r=m$, because these errors are controlled by the errors in the corresponding approximate root prior-preconditioned Hessians. In particular, the ROM-based approximate forward model $\hatFwdOperROM(r)$ described in \Cref{subsection:ROM} yields an exact solution of the inverse problem, while reducing the dimension of the inverse problem from $d$ to $m$.

\subsection{Plate}\label{sec:Plate}

In \Cref{sec:Bar}, we used a Gaussian prior with an invertible prior covariance.
In this example, we use a Gaussian prior with a singular prior covariance.
Consider a thin elastic plate in plane stress, where the displacement field $u(x,y)= \begin{pmatrix}
    u_x(x,y)\\
    u_y(x,y)
\end{pmatrix}$ solves the Navier-Lam\'{e} differential equation
\begin{flalign*}
    \mu\Delta u+(\lambda +\mu)\nabla(\nabla\cdot u)+f=0,&\\
    \mu=\frac{E}{2(1+\nu)}, \, \lambda=\frac{E\nu}{1-\nu^2}&,
\end{flalign*}
for $x\in[0,L]$ with $L=270$ and $y\in[0,H]$ with $H=928$, where $\mu,\lambda\in\R$ are the Lam\'{e} parameters, $f$ is the load, $E=210\times 10^9$ the Young's modulus and $\nu=0.3$ the Poisson ratio. The plate is fixed at the bottom, has free edges on the left and the right and is loaded by a distributed load $q(x)$ at its top edge giving boundary conditions
\begin{flalign*}
    u(x,y)|_{y=0}=0,&\\
    \sigma_{xx}|_{x=0}=0,\quad \sigma_{xy}|_{x=0}=0,&\\
    \sigma_{xx}|_{x=L}=0,\quad \sigma_{xy}|_{x=L}=0,&\\
    \sigma_{yy}|_{y=H}=q(x),\quad \sigma_{xy}|_{y=H}=0,
\end{flalign*}
where $\sigma_{xx},\sigma_{yy}$ denote the normal stress in directions normal to $x$ and $y$ surfaces while $\sigma_{xy}$ denotes the shear stress acting on a surface normal to $y$.

We discretise the plate using quadrilateral elements with linear shape functions $N_1(\xi,\eta)=1/4(1+\xi)(1+\eta)$, $N_2(\xi,\eta)=1/4(1-\xi)(1+\eta)$, $N_3(\xi,\eta)=1/4(1-\xi)(1-\eta)$ and $N_4(\xi,\eta)=1/4(1+\xi)(1-\eta)$ with elemental coordinates $\xi,\eta\in[-1,1]$. The discretisation leads to $d=1972$ degrees of freedom.

As in \Cref{sec:Bar}, the applied load is modeled as a Gaussian random field with constant mean $\mu_Q=10^6$ and exponential auto-correlation kernel with coefficient of variation $0.3$ and $\theta=0.3\times L$. The load is discretised using the midpoint method with $n=16$ random variables. The mean vector $\bm\mu_Q\in\R^n$ and covariance matrix $\mbf{\Gamma}_{QQ}\in\R^{n\times n}$ are given by $\bm\mu_Q\coloneqq \mu_Q\times\mbf{1}_n$ and $\mbf{\Gamma}_{QQ,ij}\coloneqq (0.3\mu_Q)^2\times \text{exp}\left(-\frac{\vert \Bar{x}_i-\Bar{x}_j\vert}{\theta} \right)$ for $1\le i,j \le n$ and $\Bar{x}_i$ being the midpoint of the respective element at $y=H$.
The matrix $\mbf{\Gamma}_{QQ}$ is invertible.
The nodal load vector $\unknown$ is obtained by integration over the shape functions with the map $\mbf{L}\in\R^{d\times n}$, where the contribution of element $i$ to the $e$-th component of $\unknown\in\R^d$ is $\mbf{L}_{e,i}\coloneqq \int_{A_i}N_e(\xi,\eta)\mathrm{d}\xi \mathrm{d}\eta$ for all elements $i$, and where $A_i$ is the spatial domain corresponding to the $i$-th quadrilateral element.
The matrix $\mbf{L}$ has full rank, i.e. rank $n$.
Then, $\PrMn\coloneqq \mbf{L}\bm\mu_Q\in\R^d$ and $\PrCov\coloneqq \mbf{L}\mbf{\Gamma}_{QQ}\mbf{L}^\top\in\R^{d\times d}$.
Since $d>n$, it follows that $\rank{\PrCov}=n$ and $\PrCov$ is singular.
We choose $\rootPrCov =\mbf{L} \mbf{\Gamma}_{QQ}^{1/2}$.
Since $\mbf{\Gamma}_{QQ}\in\R^{n\times n}$ is invertible and $\mbf{L}\in\R^{d\times n}$ has rank $n$, we have $\mbf{1}_n \in \range{\mbf{\Gamma}_{QQ}^{1/2}}$, and the range condition $\PrMn\in\range\rootPrCov$ in \Crefprop{proposition:OLR_approximate_forward_operator_general_prior_covariance} and \Cref{theorem:error_decomposition_approximate_root_PPH} is satisfied.

To define the observation operator $\ObsOper$ in \Crefass{assumption_main}, we drew $m=6$ locations from the uniform distribution at the top of the plate, i.e. at $y=H$ and fixed these as the locations at which to measure the state $\state$. The measurement noise follows $\mathcal{N}(0,\ObsCov)$ with $\ObsCov\coloneqq \sigma_\obs^2\times\Id_m\in\R^{m\times m}$, where $\sigma_\obs$ is set to $5\%$ of the maximum mean displacement.
We generate the data $\data$ by sampling the random variable $\dataRV$ defined according to \Crefass{assumption_main}.
For this choice of $\rootObsCov$ and for the choices described above, we have $\rankPPH\coloneqq \rank{\rootObsCov^{-1}\fwdOper\rootPrCov}=m$.
Thus, $\rootObsCov^{-1}\mbf{G}\rootPrCov$ has exactly $\rankPPH=m=6$ nonzero singular values $(\SgVal_i)_i$.
\begin{figure}
    \centering

    \begin{minipage}[b]{0.45\textwidth}
        \centering
        \includegraphics[width=\textwidth]{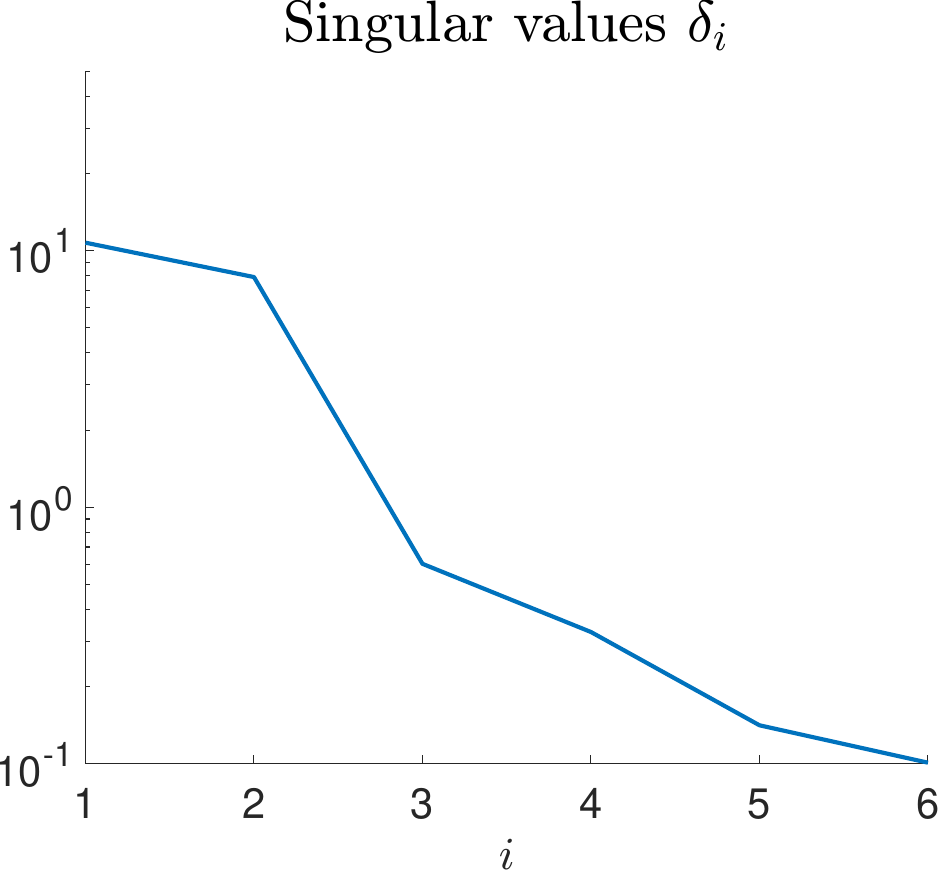}
        \small (a)
    \end{minipage}
    \hfill
    \begin{minipage}[b]{0.45\textwidth}
        \centering
        \includegraphics[width=\textwidth]{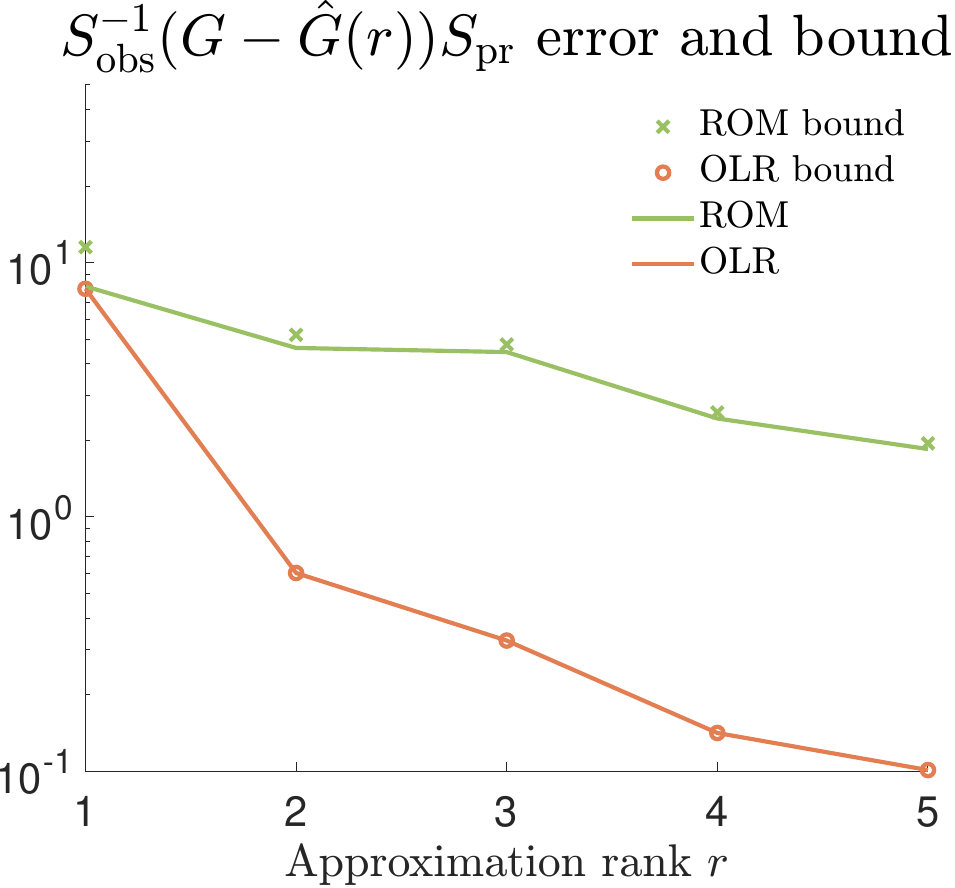}
        \small (b)
    \end{minipage}

    \medskip

    \begin{minipage}[b]{0.45\textwidth}
        \centering
        \includegraphics[width=\textwidth]{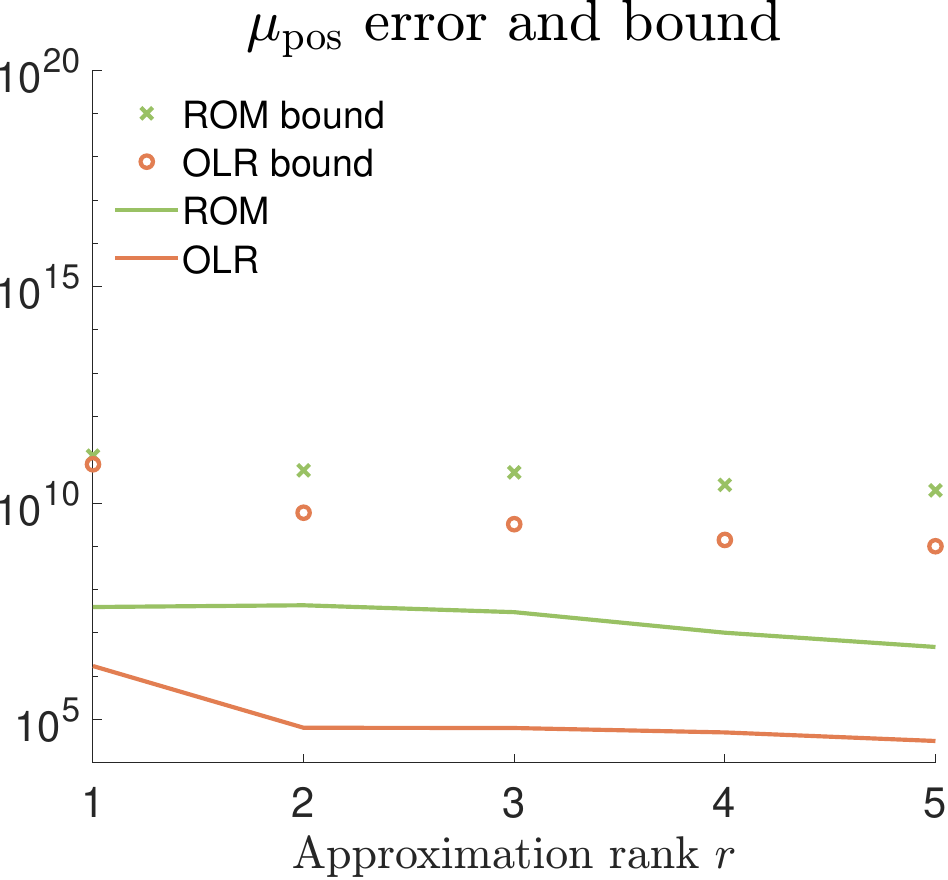}
        \small (c)
    \end{minipage}
    \hfill
    \begin{minipage}[b]{0.45\textwidth}
        \centering
        \includegraphics[width=\textwidth]{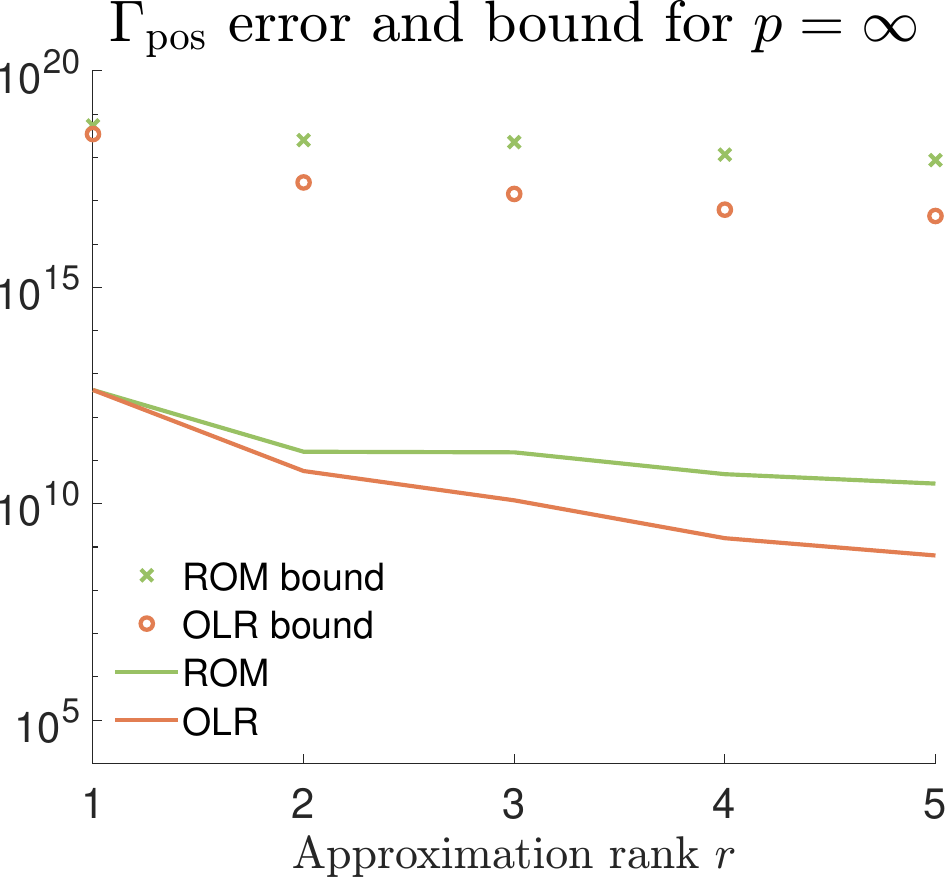}
        \small (d)
    \end{minipage}

    \caption{Comparison of $\rootObsCov^{-1}(\fwdOper-\hatFwdOper(r))\rootPrCov$, for $\hatFwdOper=\hatFwdOperOLR$ from \eqref{eq:OLR_approximate_forward_operator_general_prior_covariance} and $\hatFwdOper=\hatFwdOperROM$ from \eqref{eq:approximate_forward_operator}, for the plate problem in \Cref{sec:Plate}.
    \\
    (a): Singular values $(\SgVal_i)_i$ of $\rootObsCov^{-1}\fwdOper\rootPrCov$ from \eqref{eq:PPH_SVD}.
    \\
    (b): Errors $\Vert\rootObsCov^{-1}(\fwdOper-\hatFwdOper(r))\rootPrCov\Vert_\infty$ for $\hatFwdOper=\hatFwdOperOLR$ and $\hatFwdOper=\hatFwdOperROM$, the OLR bound $t(\infty,r)$ from \eqref{eq:OLR_errror_in_PPH}, and the ROM bound $b(\infty,r)$ from \eqref{eq:bound_on_error_of_root_PPH}.
    \\
    (c): Errors $\Vert \PosMn(\data)-\hatPosMn(\data)\Vert_2$ resulting from using $\hatFwdOperOLR$ and $\hatFwdOperROM$ in \eqref{eq:pos_mean_approx}, the OLR bound $C_2 t(\infty,r)$ from \eqref{eq:OLR_POS_Bound}, and the ROM bound $C_2^\prime b(\infty,r)$ from \eqref{eq:pos_error_ROM}.
    \\
    (d): Errors $\Vert \PosCov-\hatPosCov\Vert_\infty$ resulting from using $\hatFwdOperOLR$ and $\hatFwdOperROM$ in \eqref{eq:pos_covariance_approx}, the OLR bound $C_1 t(\infty,r)$ from \eqref{eq:OLR_POS_Bound}, and the ROM bound $C_1^\prime b(\infty,r)$ from \eqref{eq:pos_error_ROM}.}
    \label{fig:Plate}
\end{figure}

\Cref{fig:Plate}(a) shows the singular values $(\SgVal_i)_i$ of $\rootObsCov^{-1}\fwdOper\rootPrCov$ for this example.
\Cref{fig:Plate}(b) shows the errors $\Vert\rootObsCov^{-1}(\fwdOper-\hatFwdOper(r))\rootPrCov\Vert_\infty$ for $\hatFwdOper=\hatFwdOperOLR$ and $\hatFwdOper=\hatFwdOperROM$, the OLR bound $t(\infty,r)$ from \eqref{eq:OLR_errror_in_PPH}, and the ROM bound $b(\infty,r)$ from \eqref{eq:bound_on_error_of_root_PPH}. We again depict only the nonzero values, i.e., only approximations for $r<\rankPPH$.

As in \Cref{sec:Bar}, the OLR errors $\Vert\rootObsCov^{-1}(\fwdOper-\hatFwdOperOLR(r))\rootPrCov\Vert_\infty$ match the values of $(\delta_i)_{i\leq \rankPPH}$, consistent with
\eqref{eq:infty_Schatten_norm_of_error_equals_largest_trailing_singular_value}, and the analogous ROM error is larger than the OLR error. Moreover, we again observe that $t(\infty,r)$ is equal to the observed error of the OLR approximation while $b(\infty,r)$ is close but not exactly equal to the observed error of the ROM approximation; this is consistent with \eqref{eq:OLR_errror_in_PPH} and \eqref{eq:bound_on_error_of_root_PPH} respectively.
\Cref{fig:Bar}(c) and (d) show the errors $\norm{\PosMn(\data)-\hatPosMn(\data)}_2$ and $\norm{\PosCov-\hatPosCov}_\infty$ in the posterior mean and posterior covariance respectively for the ROM and OLR approximations, as well as the bounds on these errors given by \eqref{eq:OLR_POS_Bound} and \eqref{eq:pos_error_ROM}.
Due to the large values of the constants $(C_i,C_i^\prime)$ for $i=1,2$, the bounds are not tight, as per our expectations in \Crefrem{remark:tightness_of_error_bounds}.
However, the decay in the bounds tracks the decay in the errors, thus providing numerical evidence for the idea that the error in the approximate root prior-preconditioned Hessian is the key determinant for the error in the approximate posterior mean and approximate posterior covariance.
We also recall that the errors of the posterior mean and covariance approximation are equal to $0$ if $r=m$, because these errors are controlled by the errors in the corresponding approximate root prior-preconditioned Hessians, and that the ROM-based approximate forward model $\hatFwdOperROM(r)$ described in \Cref{subsection:ROM} yields an exact solution of the inverse problem, while reducing the dimension of the inverse problem from $d$ to $m$.

\section{Conclusion}
\label{section:conclusion}

This work analyses the approximation of the forward model for inverse problems with a linear forward model $\fwdOper$, Gaussian prior $\mathcal{N}(\PrMn,\PrCov)$, and Gaussian observation noise $\mathcal{N}(\mbf{0},\ObsCov)$.
In \Crefprop{proposition:optimal_low_rank_approximation}, we defined a low rank approximation $\hatFwdOperOLR$ of $\fwdOper$, and showed that the definition generalises the well-known optimal low-rank approximation from \cite{Spantini2015}---which assumes invertible prior covariances---to the context of possibly singular prior covariances.

The main theoretical analysis of this work concerns an approximation $\hatFwdOperROM$ proposed in \cite{Scheffels2025} of $\fwdOper$, where $\hatFwdOperROM$ is constructed by a reduced-order model based on Petrov-Galerkin projection of the linear parameter-to-state map associated to the PDE.
The first result in this direction, \Creflem{lemma:decomposition_of_root_of_approximate_PPH}, characterises the approximation $\hatFwdOperROM$ and shows that the corresponding root prior-preconditioned Hessian $\rootObsCov^{-1}\hatFwdOperROM(r)\rootPrCov$ agrees with the exact root prior-preconditioned Hessian $\rootObsCov^{-1} \fwdOper \rootPrCov$, when both root prior-preconditioned Hessians are projected to $\range~\lSgMat_r$, where $\lSgMat_r$ is the matrix of the $r$ leading left singular vectors of $\rootObsCov^{-1}\fwdOper\rootPrCov$.
This invariance under post-composition with $\lSgMat_r \lSgMat_r^\top$ leads to the second and main result in this direction, \Cref{theorem:error_decomposition_approximate_root_PPH}, which provides an identity for the error $\rootObsCov^{-1}(\hatFwdOperROM(r)-\fwdOper)\rootPrCov$ and a bound $b(p,r)$ for the $p$-Schatten norm of this error.
By combining this bound from a general result from \cite{KoenigLie2025}, we obtain a priori bounds on the approximate posterior covariance and approximate posterior mean associated to $\hatFwdOperROM(r)$.
In \Crefrem{remark:properties_of_bound_as_function_of_p}, we showed that the bound $b(p,r)$ vanishes when $r\geq m$.
This implies that if $r\geq m$, then $\hatFwdOperROM(r)$ identifies the low-dimensional structure in the inverse problem and yields the exact Gaussian posterior.

In \Cref{section:numerical_experiments}, we illustrated the performance of the bounds given in \Cref{theorem:error_decomposition_approximate_root_PPH} on two examples of problems that arise in the context of structural engineering, where the goal is to infer a static structural load in a boundary value problem.
Rank-deficient prior covariances arise in this context because uncertainty does not occur in every degree of freedom of the systems under consideration.
For both examples, actual errors and their a priori bounds are plotted, for both the optimal low-rank approximation $\hatFwdOperOLR$ and the reduced order model-based approximation $\hatFwdOperROM$.
The results show that $\hatFwdOperROM$ does not yield an optimal likelihood-informed ROM of $\fwdOper$, consistent with our theoretical analysis.
In addition, as the rank parameter $r$ increases, the a priori bounds closely track the decay in the observed errors, thus confirming the importance of the error in the approximate root prior-preconditioned Hessian as a determining factor of the error in the approximate posterior.

Our results emphasise that knowledge about the structure of the approximation of the forward model $\fwdOper$ can be exploited to yield tailored error bounds for approximations of the root prior-preconditioned Hessian, and consequently error bounds for the approximate posterior covariance and approximate posterior mean, in the context of linear Gaussian inverse problems.
They also show that the method in \cite{Scheffels2025} does not yield an optimal likelihood-informed model reduction method.
Potential directions for future research include the derivation of error bounds for different structure-preserving approximation methods, the development of an optimal likelihood-informed model reduction method, and structure-preserving methods for the setting where the solution or parameter-to-state operator $\fullModel^{-1}$ in \Crefass{assumption_main} may be nonlinear, as in \cite{Cvetkovic_etal_BayesOED_2024}.

\section*{Acknowledgements}
The research of HCL has been partially funded by the Deutsche Forschungsgemeinschaft (DFG) --- Project-ID 318763901 --- SFB1294 ``Data Assimilation''.
The research by JS was funded by the TUM Georg Nemetschek Institute and the Technical University of Munich - Institute for Advanced Study, Germany.

\appendix

\section{Proofs}
\label{section:proofs}
Recall the column submatrix notation from \Cref{subsection:notation}, and recall from \eqref{eq:PPH_SVD} that $\rankPPH=\rank{\rootObsCov^{-1}\fwdOper \rootPrCov}$.
\subsection{Proof for optimal low-rank approximation}
\label{subsection:proofs_OLR}

In this subsection, we prove the results from \Cref{subsection:optimal_low_rank_approximation}.
We first prove \Creflem{lemma:tilde_w_j_definitions_consistent}.
\begin{proof}
~For every $1\leq i\leq \rankPPH$,
\begin{equation*}
\rootPrCov^\top \fwdOper^\top \ObsCov^{-1} \fwdOper\widehat{\bm{w}}_i=(\rootObsCov^{-1}\fwdOper\rootPrCov)^\top  (\rootObsCov^{-1}\fwdOper\rootPrCov)\rSgVec_i=\rSgMat \SgValMat^\top \lSgMat^\top \lSgMat \SgValMat \rSgMat^\top \rSgVec_i=\SgVal_i^2 \rSgVec_i,
\end{equation*}
where the first equation follows from the hypothesis that $\widehat{\bm{w}}_i=\rootPrCov \rSgVec_i$, and the second and third equations follow from \eqref{eq:PPH_SVD}.
The hypotheses on $\PrCov$ and $\rootPrCov$ imply that $\rootPrCov$ is invertible.
Now
\begin{equation*}
 \SgVal_i^{-2}\fwdOper^\top \ObsCov^{-1} \fwdOper\widehat{\bm{w}}_i=\rootPrCov^{-\top} \rSgVec_i=\PrCov^{-1}\widehat{\bm{w}}_i,\quad 1\leq i\leq \rankPPH,
\end{equation*}
where the first equation follows by multiplying both sides of the equation $\rootPrCov^\top \fwdOper^\top \ObsCov^{-1} \fwdOper\widehat{\bm{w}}_i=\SgVal_i^2 \rSgVec_i$ from the left by $\delta_i^{-2}\rootPrCov^{-\top}$, and the second equation follows from \eqref{eq:hat_eigenvectors_and_tilde_eigenvectors_general_prior_covariances}.
For $1\leq i\leq \rankPPH$, $\SgVal_i^{-2}$ is defined, and thus the preceding display yields
\begin{align*}
    \PrCov^{-1}\widehat{\bm{w}}_i =& \fwdOper^\top\ObsCov^{-1}\fwdOper\widehat{\bm{w}}_i\SgVal_i^{-2}
    \\
    =& \fwdOper^\top(\rootObsCov^{-\top} \rootObsCov^{-1})\fwdOper\widehat{\bm{w}}_i\SgVal_i^{-2} &  \\
    =& \fwdOper^\top\rootObsCov^{-\top} \rootObsCov^{-1}\fwdOper (\rootPrCov \rSgVec_i )\SgVal_i^{-2}  & \widehat{\bm{w}}_i=\rootPrCov \rSgVec_i  \\
    =& \fwdOper^\top\rootObsCov^{-\top} \biggr( \sum_{j=1}^{d\wedge m} \SgVal_j \lSgVec_j \rSgVec_j^\top\biggr)\rSgVec_i \SgVal_i^{-2}=\fwdOper^\top\rootObsCov^{-\top}\lSgVec_i\SgVal_i^{-1} & \text{by \eqref{eq:PPH_SVD}}
\end{align*}
which completes the proof of \Creflem{lemma:tilde_w_j_definitions_consistent}.
\end{proof}

Next, we prove \Crefprop{proposition:OLR_approximate_forward_operator_general_prior_covariance}.

\begin{proof}
~We first prove \eqref{eq:OLR_approximate_forward_operator_general_prior_covariance_PPH}:
\begin{align*}
 \rootObsCov^{-1}\hatFwdOperOLR(r)\rootPrCov=&\rootObsCov^{-1}\fwdOper  \rootPrCov \rSgMat_r (\SgValMat_r^\top \SgValMat_r)^{-1/2} \lSgMat_r^\top \rootObsCov^{-1} \fwdOper\rootPrCov & \text{by \eqref{eq:testMat_and_trialMat}, \eqref{eq:OLR_approximate_forward_operator_general_prior_covariance}}
 \\
 =& \rootObsCov^{-1}\fwdOper  \rootPrCov \rSgMat_r (\SgValMat_r^\top \SgValMat_r)^{-1/2} \lSgMat_r^\top \left(\lSgMat \SgValMat \rSgMat^\top\right) & \text{by \eqref{eq:PPH_SVD}}
\\
 =&\rootObsCov^{-1}\fwdOper  \rootPrCov \rSgMat_r (\SgValMat_r^\top \SgValMat_r)^{-1/2}\underbrace{ \begin{bmatrix} \Id_r & \mbf{0}\end{bmatrix} }_{\in\R^{r\times m}} \SgValMat \rSgMat^\top
 \\
 =&\rootObsCov^{-1}\fwdOper  \rootPrCov \rSgMat_r \underbrace{\begin{bmatrix} \Id_r & \mbf{0}\end{bmatrix}}_{\in\R^{r\times n}}   \rSgMat^\top
 \\
 =&\rootObsCov^{-1}\fwdOper  \rootPrCov \rSgMat_r  \rSgMat_r^\top,
\end{align*}
where the third, fourth, and fifth equations follow from properties of $\lSgMat$, $\SgValMat$, and $\rSgMat$.
Thus
\begin{align*}
    &\Norm{\rootObsCov^{-1}(\fwdOper-\hatFwdOper(r))\rootPrCov}_p=\Norm{\rootObsCov^{-1}\fwdOper\rootPrCov-\rootObsCov^{-1}\hatFwdOper(r)\rootPrCov}_p
    \\
     =&\Norm{\rootObsCov^{-1}\fwdOper\rootPrCov-\rootObsCov^{-1}\fwdOper  \rootPrCov \rSgMat_r  \rSgMat_r^\top }_p=\Norm{\rootObsCov^{-1}\fwdOper\rootPrCov\left(\Id_d-\rSgMat_r  \rSgMat_r^\top \right)}_p & \text{by \eqref{eq:OLR_approximate_forward_operator_general_prior_covariance_PPH}}
     \\
       =& \Norm{ \lSgMat\SgValMat \rSgMat^\top \left(\Id_d-\rSgMat_r \rSgMat_r^\top \right)}_p = \Norm{ \lSgMat\SgValMat \rSgMat^\top \left(\rSgMat_\bot \rSgMat_\bot^\top \right)}_p =\Norm{\sum_{i>r} \SgVal_i \lSgVec_i\rSgVec_i^\top}_p & \text{by \eqref{eq:PPH_SVD}}\\
       =& t(p,r) & \text{by \eqref{eq:sum_trailing_singular_values}}
\end{align*}
which proves \eqref{eq:OLR_errror_in_PPH}.

Under the hypotheses that \Crefass{assumption_main} holds and $\PrMn\in\range\rootPrCov$, we may apply \cite[Theorem 3.1]{KoenigLie2025} to conclude that there exist finite scalars $C_1$ and $C_2$ such that \eqref{eq:OLR_POS_Bound} holds.

It remains to prove the final statement. Suppose $\PrCov$ is invertible and $\rootPrCov\in\R^{d\times d}$ satisfies $\PrCov=\rootPrCov\rootPrCov^\top$.
Then we may apply \Creflem{lemma:tilde_w_j_definitions_consistent}, and for $1\leq r\leq \rankPPH$,
\begin{align*}
\trialMat_r=& [\widetilde{\bm{w}}_1| \cdots |\widetilde{\bm{w}}_r] & \text{by \eqref{eq:trialMat}}
\\
=& [\PrCov^{-1}\widehat{\bm{w}}_1|\cdots |\PrCov^{-1}\widehat{\bm{w}}_r] & \text{by \Creflem{lemma:tilde_w_j_definitions_consistent}}
\\
=&\PrCov^{-1} \rootPrCov [\rSgVec_1|\cdots|\rSgVec_r] & \text{by \eqref{eq:hat_eigenvectors_and_tilde_eigenvectors_general_prior_covariances}}
\\
=& \rootPrCov^{-\top} [\rSgVec_1|\cdots|\rSgVec_r] &  \PrCov=\rootPrCov\rootPrCov^\top
\\
=& \rootPrCov^{-\top} \rSgMat_r & \text{by \eqref{eq:PPH_SVD}}
\end{align*}
and since $\testMat_r=\rootPrCov \rSgMat_r$ by \eqref{eq:testMat}, we obtain $\testMat_r \trialMat_r^\top= \rootPrCov \rSgMat_r (\rootPrCov^{-\top}\rSgMat_r)^\top$.
This completes the proof of \Crefprop{proposition:OLR_approximate_forward_operator_general_prior_covariance}.
\end{proof}

\subsection{Proofs for error analysis}
\label{subsection:proofs_error_analysis}
In this section, we prove the key results in \Cref{section:error_analysis}.
First, we prove \Creflem{lemma:decomposition_of_root_of_approximate_PPH}.
\begin{proof}
   ~Note that
\begin{align*}
 \redModel(r)\coloneqq &\trialMat_r^\top \fullModel \testMat_r &\text{by \eqref{eq:ROM_of_PDE}}
 \\
 =& \bigr(\fwdOper^\top \rootObsCov^{-\top}\lSgMat_r (\SgValMat_r^\top\SgValMat_r)^{-1/2}\bigr)^\top \fullModel \testMat_r & \text{by \eqref{eq:trialMat}}
 \\
 =&(\SgValMat_r^\top\SgValMat_r)^{-1/2} \lSgMat_r^\top \rootObsCov^{-1} \fwdOper \fullModel \testMat_r \\
 =&(\SgValMat_r^\top\SgValMat_r)^{-1/2} \lSgMat_r^\top \rootObsCov^{-1} \ObsOper \testMat_r& \text{by \eqref{eq:forward_operator}}
\end{align*}
which implies that $\redModel(r)$ is invertible for $1\leq r\leq \rankPPH$ if and only if $ \lSgMat_r^\top \rootObsCov^{-1} \ObsOper \testMat_r $ is invertible for $1\leq r \leq \rankPPH$, in which case
\begin{equation}
 \label{eq:rescaled_ROM_of_PDE}
\redModel(r)^{-1} = \left( \lSgMat_r^\top \rootObsCov^{-1} \ObsOper \testMat_r\right)^{-1} (\SgValMat_r^\top\SgValMat_r)^{1/2}.
\end{equation}
By \Crefass{assumption:ROM_of_PDE_invertible}, $\redModel(r)$ is invertible, for every $1\leq r\leq \rankPPH$.
Thus, we can rewrite the approximate forward model $\hatFwdOper(r)$ from \eqref{eq:approximate_forward_operator} as
\begin{align}
 \hatFwdOperROM(r)\coloneqq & \ObsOper \testMat_r \redModel(r)^{-1} \trialMat_r^\top  & \text{by \eqref{eq:approximate_forward_operator}}
 \nonumber
 \\
 =& \ObsOper \testMat_r \redModel(r)^{-1} \left((\SgValMat_r^\top\SgValMat_r)^{-1/2}\lSgMat_r^\top\rootObsCov^{-1}\fwdOper\right) & \text{by \eqref{eq:trialMat}}
 \nonumber
 \\
 =& \ObsOper \testMat_r \left( \lSgMat_r^\top \rootObsCov^{-1} \ObsOper \testMat_r\right)^{-1} \lSgMat_r^\top\rootObsCov^{-1}\fwdOper & \text{by \eqref{eq:rescaled_ROM_of_PDE}}
 \nonumber
\end{align}
which proves \eqref{eq:approximate_forward_operator_ROM}.
In turn, \eqref{eq:approximate_forward_operator_ROM} implies \eqref{eq:approximate_forward_operator_ROM_post_projection_property}:
\begin{align*}
  \lSgMat_r \lSgMat_r^\top \rootObsCov^{-1}\hatFwdOperROM(r)\rootPrCov=& \lSgMat_r \left(\lSgMat_r^\top \rootObsCov^{-1}\ObsOper \testMat_r \right)\left( \lSgMat_r^\top \rootObsCov^{-1} \ObsOper \testMat_r\right)^{-1} \lSgMat_r^\top \rootObsCov^{-1}\fwdOper\rootPrCov
  \\
  =& \lSgMat_r  \lSgMat_r^\top \rootObsCov^{-1}\fwdOper\rootPrCov.
 \end{align*}
 This completes the proof of \Creflem{lemma:decomposition_of_root_of_approximate_PPH}.
\end{proof}

We now prove \Cref{theorem:error_decomposition_approximate_root_PPH}.
\begin{proof}
~Under the hypotheses of \Cref{theorem:error_decomposition_approximate_root_PPH}, we can apply \Creflem{lemma:decomposition_of_root_of_approximate_PPH}.
Since $\lSgMat=[\lSgMat_r | \lSgMat_\bot]\in\R^{m\times m}$ in \eqref{eq:PPH_SVD} is orthogonal, we have $\Id_m=\lSgMat_r\lSgMat_r^\top +\lSgMat_\bot \lSgMat_\bot^\top$.
Thus
 \begin{align*}
  &\rootObsCov^{-1}\fwdOper \rootPrCov-\rootObsCov^{-1}\hatFwdOperROM(r) \rootPrCov
  \\
  =&\left(\lSgMat_r\lSgMat_r^\top +\lSgMat_\bot \lSgMat_\bot^\top\right)\left( \rootObsCov^{-1}\fwdOper \rootPrCov-\rootObsCov^{-1}\hatFwdOperROM(r) \rootPrCov\right)
  \\
  =&\lSgMat_\bot \lSgMat_\bot^\top \left( \rootObsCov^{-1}\fwdOper \rootPrCov-\rootObsCov^{-1}\hatFwdOperROM(r) \rootPrCov\right) & \text{by \eqref{eq:approximate_forward_operator_ROM_post_projection_property}}
  \\
  =&\lSgMat_\bot \lSgMat_\bot^\top\rootObsCov^{-1}\fwdOper \rootPrCov
  \\
  &-\lSgMat_\bot \lSgMat_\bot^\top\rootObsCov^{-1}\ObsOper \testMat_r \left( \lSgMat_r^\top \rootObsCov^{-1} \ObsOper \testMat_r\right)^{-1} \lSgMat_r^\top\rootObsCov^{-1}\fwdOper\rootPrCov & \text{by \eqref{eq:approximate_forward_operator_ROM}}
 \end{align*}
 which proves \eqref{eq:decomposition_of_error_of_root_PPH}.
 By \eqref{eq:PPH_SVD}, $\rootObsCov^{-1}\fwdOper\rootPrCov=\lSgMat_\bot(\SgValMat_\bot^\top \SgValMat_\bot)^{1/2}\rSgMat_\bot^\top+\lSgMat_r(\SgValMat_r^\top \SgValMat_r)^{1/2}\rSgMat_r^\top$.
 Substituting this into \eqref{eq:decomposition_of_error_of_root_PPH} yields \eqref{eq:decomposition_of_error_of_root_PPH_alternate_form}:
\begin{align*}
&  \rootObsCov^{-1}\fwdOper \rootPrCov-\rootObsCov^{-1}\hatFwdOperROM(r)\rootPrCov
\nonumber
\\
  =& \lSgMat_\bot (\SgValMat_\bot^\top \SgValMat_\bot)^{1/2} \rSgMat_\bot^\top
  -\lSgMat_\bot \lSgMat_\bot^\top\rootObsCov^{-1}\ObsOper \testMat_r \left( \lSgMat_r^\top \rootObsCov^{-1} \ObsOper \testMat_r\right)^{-1}(\SgValMat_r^\top\SgValMat_r)^{1/2}\rSgMat_r^\top.
\end{align*}
   By \eqref{eq:PPH_SVD} and the definition \eqref{eq:sum_trailing_singular_values} of $t(p,r)$,
  \begin{equation*}
   \Norm{ \lSgMat_\bot (\SgValMat_\bot^\top \SgValMat_\bot)^{1/2} \rSgMat_\bot^\top }_p=\biggr\Vert \sum_{i\geq r+1} \SgVal_i \lSgVec_i \rSgVec_i^\top\biggr\Vert_p=\norm{(\SgVal_i)_{i\geq r+1}}_{p}=t(p,r).
  \end{equation*}
  By applying the $p$-Schatten norm to both sides of \eqref{eq:decomposition_of_error_of_root_PPH_alternate_form}, the triangle inequality, and the definition \eqref{eq:bound_on_error_of_root_PPH} of $b(p,r)$,
\begin{align*}
 &\Norm{\rootObsCov^{-1}\fwdOper \rootPrCov-\rootObsCov^{-1}\hatFwdOper(r)\rootPrCov}_p
 \\
 \leq & t(p,r)+\Norm{\lSgMat_\bot \lSgMat_\bot^\top\rootObsCov^{-1}\ObsOper \testMat_r \left( \lSgMat_r^\top \rootObsCov^{-1} \ObsOper \testMat_r\right)^{-1}(\SgValMat_r^\top\SgValMat_r)^{1/2}\rSgMat_r^\top}_p\eqqcolon b(p,r),
  \end{align*}
thus proving \eqref{eq:bound_on_error_of_root_PPH}.

It remains to prove the bounds in \eqref{eq:pos_error_ROM}.
Under the hypotheses, we may apply \cite[Theorem 3.1]{KoenigLie2025} to conclude that for any $\data$, the stated scalars $C^\prime_i$, $i=1,2$ exist and satisfy
\begin{align*}
 \Norm{\PosCov-\hatPosCov}_p\leq& C^\prime_1 \norm{\rootObsCov^{-1}(\fwdOper-\hatFwdOper)\rootPrCov}_p
 \\
 \Norm{\PosMn(\data)-\hatPosMn(\data)}_2\leq &C^\prime_2 \norm{\rootObsCov^{-1}(\fwdOper-\hatFwdOper)\rootPrCov}_\infty.
\end{align*}
By the bound \eqref{eq:bound_on_error_of_root_PPH} from \Cref{theorem:error_decomposition_approximate_root_PPH}, the bounds on the approximate posterior covariance and approximate posterior mean in \eqref{eq:pos_error_ROM} follow.
This completes the proof of \Cref{theorem:error_decomposition_approximate_root_PPH}.
\end{proof}

\bibliographystyle{plainnat}
\bibliography{biblio}

@article{Aziznejad26042021,
author = {Aziznejad, Shayan and Unser, Michael},
title = {Duality Mapping for {S}chatten Matrix Norms},
fjournal = {Numerical Functional Analysis and Optimization},
journal = {Numer. Funct. Anal. Optim.},
volume = {42},
number = {6},
pages = {679--695},
year = {2021},
doi = {10.1080/01630563.2021.1922438},
}

@misc {4709579,
    TITLE = {Case of equality in spectral norm matrix triangle inequality},
    NOTE = {URL: https://math.stackexchange.com/q/4709579 (Accessed: 2026-05-11)},
    URL = {https://math.stackexchange.com/q/4709579}
}

@unpublished{Scheffels2025,
title		= {{Likelihood-informed model reduction for Bayesian inference of static structural loads}},
author		= {Scheffels, Jakob and Qian, Elizabeth and Papaioannou, Iason and Ullmann, Elisabeth},
year		= {2026},
doi         = {10.48550/arXiv.2510.07950},
note        = {arXiv:2510.07950. To appear in Struct. Multidiscip. Optim.}
}

@unpublished{CarereLie2025a,
title       = {Optimal low-rank posterior covariance approximation in linear {Gaussian} inverse problems on {Hilbert} spaces},
author      = {Carere, G. and Lie, H. C.},
year        = {2025},
doi 		= {10.48550/arXiv.2411.01112},
note        = {arXiv:2411.01112}}

@article{CarereLie2025b,
title = {{Optimal low-rank posterior mean and distribution approximation in linear Gaussian inverse problems on Hilbert spaces}},
author={Giuseppe Carere and Han Cheng Lie},
fjournal = {Inverse Problems and Imaging},
journal = {Inverse Probl. Imag.},
pages = {},
year = {2026},
doi = {10.3934/ipi.2026031},
}

@article{Flath2011,
author = {Flath, H. P. and Wilcox, L. C. and Akcelik, V. and Hill, J. and Van Bloemen Waanders, B. and Ghattas, O.},
title = {Fast algorithms for {B}ayesian uncertainty quantification in large-scale linear inverse problems based on low-rank partial {H}essian approximations},
fjournal = {SIAM Journal on Scientific Computing},
journal = {SIAM J. Sci. Comput.},
volume = {33},
number = {1},
pages = {407-342},
year = {2011},
doi = {10.1137/090780717},
}

@article{Spantini2015,
author      = {Spantini, Alessio and Solonen, Antti and Cui, Tiangang and Martin, James and Tenorio, Luis and Marzouk, Youssef},
title       = {Optimal Low-rank Approximations of {Bayesian} Linear Inverse Problems},
journal     = {SIAM J. Sci. Comput.},
volume      = {37},
number      = {6},
pages       = {A2451-A2487},
year        = {2015},
doi         = {10.1137/140977308}}

@article{eiermannGeometricAspectsTheory2001,
	title = {Geometric aspects of the theory of {Krylov} subspace methods},
	volume = {10},
	issn = {1474-0508, 0962-4929},
	doi = {10.1017/S0962492901000046},
	journal = {Acta Numerica},
	author = {Eiermann, Michael and Ernst, Oliver G.},
	year = {2001},
	pages = {251--312},
}

@unpublished{KoenigLie2025,
title       = {Posterior error bounds for prior-driven balancing}, 
author      = {K\"{o}nig, Josie and Lie, Han Cheng},
year        = {2025},
doi         = {10.48550/arXiv.2601.03971},
note        = {arXiv:2601.03971}
}

@article{Cvetkovic_etal_BayesOED_2024,
author = {Cvetkovi\'{c}, Nada and Lie, Han Cheng and Bansal, Harshit and Veroy, Karen},
title = {Choosing Observation Operators to Mitigate Model Error in {B}ayesian Inverse Problems},
fjournal = {SIAM/ASA Journal on Uncertainty Quantification},
journal = {SIAM/ASA J. Uncertain. Quantif.},
volume = {12},
number = {3},
pages = {723-758},
year = {2024},
doi = {10.1137/23M1602140},
}

@article{der1988stochastic,
  title={The stochastic finite element method in structural reliability},
  author={Der Kiureghian, Armen and Ke, Jyh-Bin},
  journal={PEM},
  volume={3},
  number={2},
  pages={83--91},
  doi = {10.1016/0266-8920(88)90019-7},
  year={1988},
  publisher={Elsevier}
}

@unpublished{konigDimensionModelReduction2025,
	title = {Dimension and model reduction approaches for linear {Bayesian} inverse problems with rank-deficient prior covariances},
	doi = {10.48550/arXiv.2506.23892},
	publisher = {arXiv},
	author = {König, Josie and Qian, Elizabeth and Freitag, Melina A.},
	year = {2025},
    note = {arXiv:2506.23892}
}

@article{stuartInverseProblemsBayesian2010b,
	title = {Inverse problems: {A} {Bayesian} perspective},
	volume = {19},
	doi = {10.1017/S0962492910000061},
	journal = {Acta Numerica},
	author = {Stuart, A. M.},
	year = {2010},
	pages = {451--559},
}

@article{sprungkLocalLipschitzStability2020a,
	title = {On the local {Lipschitz} stability of {Bayesian} inverse problems},
	volume = {36},
	doi = {10.1088/1361-6420/ab6f43},
	journal = {Inverse Problems},
	publisher = {IOP Publishing},
	author = {Sprungk, Björn},
	year = {2020},
	pages = {055015},
}

@article{AdaptiveMultifidelityPCE2019,
	title = {Adaptive multi-fidelity polynomial chaos approach to {Bayesian} inference in inverse problems},
	volume = {381},
	doi = {10.1016/j.jcp.2018.12.025},
	journal = {J. Comput. Phys.},
	author = {Yan, Liang and Zhou, Tao},
	year = {2019},
	pages = {110--128},
}

@article{novakPhysicsinformedPolynomialChaos2024,
	title = {Physics-informed polynomial chaos expansions},
	volume = {506},
	doi = {10.1016/j.jcp.2024.112926},
	journal = {J. Comput. Phys.},
	author = {Novák, Lukáš and Sharma, Himanshu and Shields, Michael D.},
	year = {2024},
	pages = {112926},
}

@article{LimitationsPolynomialChaos2015,
	title = {Limitations of polynomial chaos expansions in the {Bayesian} solution of inverse problems},
	volume = {282},
	doi = {10.1016/j.jcp.2014.11.010},
	journal = {J. Comput. Phys.},
	author = {Lu, Fei and Morzfeld, Matthias and Tu, Xuemin and Chorin, Alexandre J.},
	year = {2015},
	pages = {138--147}
}

@article{dinkelSolvingBayesianInverse2023,
  title={Solving {B}ayesian inverse problems with expensive likelihoods using constrained {G}aussian processes and active learning},
  author={Dinkel, Maximilian and Geitner, Carolin M and Rei, Gil Robalo and Nitzler, Jonas and Wall, Wolfgang A},
  journal={Inverse Problems},
  volume={40},
  number={9},
  pages={095008},
  year={2024},
  publisher={IOP Publishing},
  doi = {10.1088/1361-6420/ad5eb4},
}

@unpublished{villaniAdaptiveGaussianProcess2024,
	title = {Adaptive {Gaussian} {process} {regression} for {Bayesian} inverse problems},
	doi = {10.48550/arXiv.2404.19459},
	publisher = {arXiv},
	author = {Villani, Paolo and Unger, Jörg and Weiser, Martin},
	year = {2024},
    note = {arXiv:2404.19459}
}

@article{deveneyDeepSurrogateApproach2021,
	title = {Deep surrogate accelerated delayed-acceptance {H}amiltonian {M}onte {C}arlo for {B}ayesian inference of spatio-temporal heat fluxes in rotating disc systems},
	doi = {https://doi.org/10.1137/22M1513113},
	journal = {SIAM-ASA J. Uncertain. Quantif.},
	author = {Deveney, Teo and Mueller, Eike and Shardlow, Tony},
	year = {2023},
    pages={970-995},
    volume={11}
}

@article{yanAdaptiveSurrogateModeling2020,
	title = {An adaptive surrogate modeling based on deep neural networks for large-scale {Bayesian} inverse problems},
	doi = {10.4208/cicp.OA-2020-0186},
	journal = {Commun. Comput. Phys.},
	author = {Yan, Liang},
	year = {2020},
    volume = {28}
}

@article{pasparakisBayesianNeuralNetworks2025,
	title = {Bayesian neural networks for predicting uncertainty in full-field material response},
	volume = {433},
	doi = {10.1016/j.cma.2024.117486},
	journal = {Comput. Methods Appl. Mech. Eng.},
	author = {Pasparakis, George D. and Graham-Brady, Lori and Shields, Michael D.},
	year = {2025},
	pages = {117486},
}

@unpublished{pfortnerPhysicsInformedGaussianProcess2024,
  title={Physics-informed {G}aussian process regression generalizes linear {PDE} solvers},
  author={Pf{\"o}rtner, Marvin and Steinwart, Ingo and Hennig, Philipp and Wenger, Jonathan},
  journal={arXiv preprint arXiv:2212.12474},
  year={2022},
  doi = {10.48550/arXiv.2212.12474},
  note = {arXiv:2212.12474}
}

@book{Antoulas2005,
  author    = {Athanasios C. Antoulas},
  title     = {Approximation of large-scale dynamical systems},
  year      = {2005},
  publisher = {SIAM},
  address   = {Philadelphia},
  series    = {Advances in Design and Control},
  volume    = {6},
  doi       = {10.1137/1.9780898718713},
}

@article{benner2015survey,
  title={A survey of projection-based model reduction methods for parametric dynamical systems},
  author={Benner, Peter and Gugercin, Serkan and Willcox, Karen},
  journal={SIAM review},
  volume={57},
  number={4},
  pages={483--531},
  year={2015},
  doi={10.1137/130932715},
  publisher={SIAM}
}

@INPROCEEDINGS{bui-thanhExtremescaleUQBayesian2012,
  author={Bui-Thanh, Tan and Burstedde, Carsten and Ghattas, Omar and Martin, James and Stadler, Georg and Wilcox, Lucas C.},
  booktitle={SC '12: Proc. Int. Conf. High Perform. Comput. Netw. Storage Anal.}, 
  title={Extreme-scale {UQ} for {B}ayesian inverse problems governed by {PDE}s}, 
  year={2012},
  volume={},
  number={},
  pages={1-11},
  doi={10.1109/SC.2012.56}}

@article{bui-thanhComputationalFrameworkInfiniteDimensional2013,
	title = {A {computational} {framework} for {infinite}-{dimensional} {Bayesian} {inverse} {problems} {part} {I}: {the} {linearized} {case}, with {application} to {global} {seismic} {inversion}},
	volume = {35},
	doi = {10.1137/12089586X},
	number = {6},
	journal = {J. Sci. Comput.},
	author = {Bui-Thanh, Tan and Ghattas, Omar and Martin, James and Stadler, Georg},
	year = {2013},
	pages = {A2494--A2523}
}

@article{cuiLikelihoodinformedDimensionReduction2014b,
	title = {Likelihood-informed dimension reduction for nonlinear inverse problems},
	volume = {30},
	doi = {10.1088/0266-5611/30/11/114015},
	number = {11},
	journal = {Inverse Problems},
	author = {Cui, T and Martin, J and Marzouk, Y M and Solonen, A and Spantini, A},
	year = {2014},
	pages = {114015},
}

@article{zahm2022certified,
  title={Certified dimension reduction in nonlinear {B}ayesian inverse problems},
  author={Zahm, Olivier and Cui, Tiangang and Law, Kody and Spantini, Alessio and Marzouk, Youssef},
  journal={Math. Comp.},
  volume={91},
  number={336},
  pages={1789--1835},
  doi={10.1090/mcom/3737},
  year={2022}
}

@article{qianModelReductionLinearBalancing,
	title = {Model {Reduction} of {Linear} {Dynamical} {Systems} via {Balancing} for {Bayesian} {Inference}},
	volume = {91},
	doi = {10.1007/s10915-022-01798-8},
	number = {1},
	journal = {J. Sci. Comput.},
	author = {Qian, Elizabeth and Tabeart, Jemima M. and Beattie, Christopher and Gugercin, Serkan and Jiang, Jiahua and Kramer, Peter R. and Narayan, Akil},
	year = {2022},
	pages = {29},
}

@article{konigTimeLimitedBalancedTruncation2023,
	title = {Time-{Limited} {Balanced} {Truncation} for {Data} {Assimilation} {Problems}},
	volume = {97},
	doi = {10.1007/s10915-023-02358-4},
	number = {2},
	journal = {J. Sci. Comput.},
	author = {König, Josie and Freitag, Melina A.},
	year = {2023},
	pages = {47},
}

@unpublished{stavrinidesEnsembleKalmanApproach2025,
	title = {An ensemble {Kalman} approach to randomized maximum likelihood estimation},
	doi = {10.48550/arXiv.2507.03207},
	publisher = {arXiv},
	author = {Stavrinides, Pavlos and Qian, Elizabeth},
	year = {2025},
    note = {arXiv:2507.03207}
}

@article{freitagInferenceOrientedBalancedTruncation2024a,
	title = {Inference-{Oriented} {Balanced} {Truncation} for {Quadratic} {Dynamical} {Systems}: {Formulation} for {Bayesian} {Smoothing} and {Model} {Stability} {Analysis}},
	volume = {24},
	doi = {10.1002/pamm.202400051},
	number = {4},
	journal = {PAMM},
	author = {Freitag, Melina A. and König, Josie and Qian, Elizabeth},
	year = {2024},
	pages = {e202400051},
}

@incollection{lumley1981coherent,
	title = {Coherent {Structures} in {Turbulence}},
	booktitle = {Transition and {Turbulence}},
	publisher = {Academic Press},
	author = {Lumley, J. L.},
	month = jan,
	year = {1981},
	doi = {10.1016/B978-0-12-493240-1.50017-X},
	pages = {215--242},
}

@article{sirovich1987turbulence,
  title={Turbulence and the dynamics of coherent structures. {I. C}oherent structures},
  author={Sirovich, Lawrence},
  journal={Quart. Appl. Math.},
  volume={45},
  number={3},
  pages={561--571},
  doi={10.1090/qam/910462},
  year={1987}
}

@article{ghattasLearningPhysicsbasedModels2021a,
	title = {Learning physics-based models from data: perspectives from inverse problems and model reduction},
	volume = {30},
	doi = {10.1017/S0962492921000064},
	journal = {Acta Numerica},
	author = {Ghattas, Omar and Willcox, Karen},
	year = {2021},
	pages = {445--554},
}

@article{nguyenModelOrderReduction2014,
	title = {Model order reduction for {Bayesian} approach to inverse problems},
	volume = {1},
	doi = {10.1186/2196-1166-1-2},
	number = {1},
	journal = {Asia Pac. J. Comput. Eng.},
	author = {Nguyen, Ngoc-Hien and Cheong Khoo, Boo and Willcox, Karen},
	year = {2014},
	pages = {2},
}

@article{cuiDatadrivenModelReduction2015,
	author = {Cui, Tiangang and Marzouk, Youssef M. and Willcox, Karen},
	title = {Data-driven model reduction for the {Bayesian} solution of inverse problems},
	journal = {Int. J. Numer. Meth. Engng},
	volume = {102},	
    number = {5},
	pages = {966--990},
	year = {2015},
	doi = {10.1002/nme.4748}
}

@article{xiongAcceleratingBayesianInference2021,
	title = {Accelerating the {Bayesian} inference of inverse problems by using data-driven compressive sensing method based on proper orthogonal decomposition},
	volume = {29},
	doi = {10.3934/era.2021044},
	number = {5},
	journal = {Electron. Res. Arch.},
	author = {Xiong, Meixin and Chen, Liuhong and Ming, Ju and Shin, Jaemin and Xiong, Meixin and Chen, Liuhong and Ming, Ju and Shin, Jaemin},
	year = {2021},
	pages = {3383--3403},
}

@article{raoInverseParameterEstimation2024,
	title = {Inverse parameter estimation using compressed sensing and {POD}-{RBF} reduced order models},
	volume = {422},
	doi = {10.1016/j.cma.2024.116820},
	journal = {Comput. Methods Appl. Mech. Eng.},
	author = {Rao, Preetham P.},
	year = {2024},
	pages = {116820},
}

@article{hesthaven2022reduced,
  title={Reduced basis methods for time-dependent problems},
  author={Hesthaven, Jan S and Pagliantini, Cecilia and Rozza, Gianluigi},
  journal={Acta Numerica},
  volume={31},
  pages={265--345},
  year={2022},
  doi={10.1017/S09624929220000580},
  publisher={Cambridge University Press}
}

@article{rozza2024short,
  title={Real time reduced order computational mechanics},
  author={Rozza, Gianluigi and Ballarin, Francesco and Scandurra, Leonardo and Pichi, Federico},
  journal={SISSA springer series, Springer Cham},
  year={2024},
  doi={10.1007/978-3-031-49892-3},
  publisher={Springer},
  volume = {5}
}

@article{chenSteinVariationalReduced2021,
	title = {Stein {variational} {reduced} {basis} {Bayesian} {inversion}},
	volume = {43},
	doi = {10.1137/20M1321589},
	number = {2},
	journal = {J. Sci. Comput.},
	author = {Chen, Peng and Ghattas, Omar},
	year = {2021},
	pages = {A1163--A1193},
}

@article{silvaReducedBasisEnsemble2023,
	title = {A reduced basis ensemble {Kalman} method},
	volume = {14},
	doi = {10.1007/s13137-023-00235-8},
	
	number = {1},
	journal = {GEM Int. J. Geomath.},
	author = {Silva, Francesco A. B. and Pagliantini, Cecilia and Grepl, Martin and Veroy, Karen},
	year = {2023},
	pages = {24},
}

@article{cotterApproximationBayesianInverse2010a,
	title = {Approximation of {Bayesian} {Inverse} {Problems} for {PDEs}},
	volume = {48},
	doi = {10.1137/090770734},
	journal = {SIAM Journal on Numerical Analysis},
	publisher = {Society for Industrial and Applied Mathematics},
	author = {Cotter, S. L. and Dashti, M. and Stuart, A. M.},
	year = {2010},
	pages = {322--345},
}

\end{document}